\documentclass{amsart}
\usepackage{amsfonts,amssymb,amscd,amsmath,latexsym,amsbsy,bm}
\numberwithin{equation}{section}
\theoremstyle{plain}
\usepackage{array}
\usepackage{tikz}

\newtheorem{thm}{Theorem}[section]

\newtheorem{prop}[thm]{Proposition}
\newtheorem{defi}[thm]{Definition}
\newtheorem{lem}[thm]{Lemma}

\theoremstyle{remark}
\newtheorem{rema}[thm]{Remark}

\makeatletter
\newcommand*\rel@kern[1]{\kern#1\dimexpr\macc@kerna}
\newcommand*\widebar[1]{%
  \begingroup
  \def\mathaccent##1##2{%
    \rel@kern{0.8}%
    \overline{\rel@kern{-0.8}\macc@nucleus\rel@kern{0.2}}%
    \rel@kern{-0.2}%
  }%
  \macc@depth\@ne
  \let\math@bgroup\@empty \let\math@egroup\macc@set@skewchar
  \mathsurround\z@ \frozen@everymath{\mathgroup\macc@group\relax}%
  \macc@set@skewchar\relax
  \let\mathaccentV\macc@nested@a
  \macc@nested@a\relax111{#1}%
  \endgroup
}
\makeatother

\newcommand{\Z}{\mathbb{Z}}

\newcommand{\C}{\mathbb{C}}
\newcommand{\e}{\mathrm{e}}
\renewcommand{\d}{\mathrm{d}}

\newcommand{\id}{\mathrm{id}}
\newcommand{\Id}{\mathrm{Id}}
\newcommand{\End}{\mathrm{End}}

\newcommand{\pr}{\mathrm{pr}}
\newcommand{\siy}{\sigma^{\mathrm{y}}}
\renewcommand{\c}{\mathrm{c}}

\title[Integral solutions to boundary qKZ equations]{Integral solutions to boundary quantum Knizhnik-Zamolodchikov equations}
\author{Nicolai Reshetikhin, Jasper Stokman, Bart Vlaar}
\address{N.R.: Department of Mathematics, University of California, Berkeley,
CA 94720, USA \& ITMO University, Kronverskii Ave. 49, Saint Petersburg, 197101, Russia \& KdV Institute for Mathematics, University of Amsterdam, Science Park 105-107, 1098 XG Amsterdam, The Netherlands}
\email{reshetik@math.berkeley.edu}
\address{J.S.: KdV Institute for Mathematics, University of Amsterdam,
Science Park 105-107, 1098 XG Amsterdam, The Netherlands.}
\email{j.v.stokman@uva.nl}
\address{B.V.: School of Mathematical Sciences, University of Nottingham, University Park, Nottingham, NG7 2RD, UK.}
\email{Bart.Vlaar@nottingham.ac.uk}

\begin{document}
\maketitle
\setcounter{tocdepth}{1}
\begin{abstract}
We construct integral representations of solutions to the boundary quantum Knizhnik-Zamolodchikov equations.
These are difference equations taking values in tensor products of Verma modules of quantum affine $\mathfrak{sl}_2$, with the K-operators acting diagonally. 
The integrands in question are products of scalar-valued elliptic weight functions with 
vector-valued trigonometric weight functions (boundary Bethe vectors).
These integrals give rise to a basis of solutions of the boundary 
qKZ equations over the field of quasi-constant meromorphic functions in weight subspaces of the tensor product. 
\end{abstract}


\section{Introduction}

The boundary quantum Knizhnik-Zamolodchikov (qKZ) equations have roots in representation theory, in the theory of solvable models in statistical mechanics and integrable quantum field theory. 
They first appeared in \cite{Ch} where the qKZ equation of type A \cite{FR} was generalized to other affine Weyl types. 
What we call here the boundary qKZ equation corresponds to type C. 
Shortly after the work \cite{Ch} the boundary qKZ equation appeared in \cite{JKKMW} as an equation for correlation functions in the 6-vertex model with reflecting boundary conditions \cite{Sk}.
The boundary qKZ equations also appear as an equation for form-factors in such models and in quantum integrable field theories on the half-line \cite{GZ}.
In \cite{Vl} the precise relation between the boundary qKZ equations and Sklyanin's commuting transfer matrices \cite{Sk} was established.

Recall that the boundary qKZ equations related to quantum affine $\mathfrak{sl}_2$ require the following data (we refer the reader to Section \ref{sec:integrabledata} for more details).
Fix highest weights $\ell_1, \dots, \ell_N \in \C$ of Verma modules $V^{\ell_1},\ldots,V^{\ell_N}$ of $U_q(\widehat{\mathfrak{sl}}_2)$ and a step size 
$\tau \in \C$. 
Recall that the Yang-Baxter relation for linear operators $R^{\ell_1\ell_2}(x): V^{\ell_1}\otimes V^{\ell_2}\rightarrow V^{\ell_1}\otimes V^{\ell_2}$ is the identity
\begin{equation}\label{YBE}
R^{\ell_1 \ell_2}_{12}(x) R^{\ell_1 \, \ell_3}_{13}(x+y) R^{\ell_2 \, \ell_3}_{23}(y) =  R^{\ell_2 \,\ell_3}_{23}(y) R^{\ell_1 \, \ell_3}_{13}(x+y) R^{\ell_1 \ell_2}_{12}(x),
\end{equation}
for linear operators in $V^{\ell_1} \otimes V^{\ell_2} \otimes V^{\ell_3}$. 
Here we use standard tensor leg notations, e.g., 
$R^{\ell_1 \ell_2}_{12}(x) =R^{\ell_1 \ell_2}(x)\otimes \Id_{V^{\ell_3}}$. 
Given $R^{\ell_1 \ell_2}(x)$ satisfying \eqref{YBE}, 
the left and right reflection equations are identities in 
$V^{\ell_1} \otimes V^{\ell_2}$ for two linear operators 
$K^{\pm,\ell}(x): V^\ell \to V^\ell$ ($\ell=\ell_1,\ell_2$), namely
\begin{equation} \label{REs}
\begin{aligned}
&R^{\ell_1 \ell_2}(x-y) K^{+,\ell_1}_{1}(x) P^{\ell_2 \ell_1} R^{\ell_2 \ell_1}(x+y) P^{\ell_1 \ell_2} K^{+,\ell_2}_{2}(y) = \\
&\quad= K^{+,\ell_2}_{2}(y) R^{\ell_1 \ell_2}(x+y) K^{+,\ell_1}_{1}(x) P^{\ell_2 \ell_1}  R^{\ell_2 \ell_1}(x-y) P^{\ell_1 \ell_2}, \\
& P^{\ell_2 \ell_1}  R^{\ell_2 \ell_1}(x-y) P^{\ell_1 \ell_2} K^{-,\ell_1}_{1}(x) R^{\ell_1 \ell_2}(x+y) K^{-,\ell_2}_2(y) = \\
&\quad= K^{-,\ell_2}_{2}(y)  P^{\ell_2 \ell_1} R^{\ell_2 \ell_1}(x+y) P^{\ell_1 \ell_2} K^{-,\ell_1}_{1}(x) R^{\ell_1 \ell_2}(x-y),
\end{aligned}
\end{equation}
respectively; here $P^{\ell_1 \ell_2}: V^{\ell_1} \otimes V^{\ell_2} \to V^{\ell_2}  \otimes V^{\ell_1}$ is the flip of tensor factors.

The corresponding boundary qKZ equations are the compatible system of
difference equations
\begin{equation} \label{bqkz-int}
\begin{aligned}
& \Psi(t_1, \dots, t_r+\tau, \dots, t_N) =
R^{\ell_r \, \ell_{r+1}}_{r \, r+1}(t_r-t_{r+1}+\tau) \cdots R^{\ell_r \, \ell_N}_{r \, N}(t_r-t_N+\tau)  \\
 & \quad \times K^{+,\ell_r}_r(t_r+\tfrac{\tau}{2})  R^{\ell_N \, \ell_r}_{N \, r}(t_N+t_r) \cdots R^{\ell_{r+1} \, \ell_r}_{r+1 \, r}(t_{r+1}+t_r) \\
 & \quad \times R^{\ell_{r-1} \, \ell_r}_{r-1 \, r}(t_{r-1}+t_r) \cdots R^{\ell_1 \, \ell_r}_{1 \, r}(t_1+t_r) K^{-,\ell_r}_r(t_r) \\
 & \quad \times \bigl( R^{\ell_1 \, \ell_r}_{1 \, r}(t_1-t_r) \bigr)^{-1} \cdots \bigl(R^{\ell_{r-1} \, \ell_r}_{r-1 \, r}(t_{r-1}-t_r) \bigr)^{-1}
 \Psi(t_1, \dots, t_r, \dots, t_N),\hspace{-15pt} \end{aligned}
\end{equation}
with $r \in \{1,\ldots,N\}$ for meromorphic $V^{\ell_1}\otimes\cdots\otimes 
V^{\ell_N}$-valued functions 
$\Psi(\bm t)$ in $\bm t=(t_1,\ldots,t_N)\in\mathbb{C}^N$. 
The Yang-Baxter equation \eqref{YBE} and the reflection equations \eqref{REs} 
guarantee the compatibility of the system \eqref{bqkz-int}. \\

Special solutions to the boundary qKZ equations are known. 
First results go back to the papers \cite{JKKKM,JKKMW} where the Heisenberg algebra realization of q-vertex operators was used to construct a specific solution corresponding to correlation functions in the 6-vertex model. 
This method was applied to other models, to obtain special solutions to boundary qKZ for other R-matrices, see \cite{Ko} for an overview of some results in this direction. 
A family of solutions to the boundary qKZ equations was constructed in our earlier papers \cite{RSV,RSV2} as Jackson integrals (bilateral series), cf. the works \cite{R,TV1} which dealt with the same topic in type A. 
In \cite{SV} Laurent-polynomial solutions were found in terms of nonsymmetric Koornwinder polynomials associated to principal series modules of the affine Hecke algebra of type C.

\subsection{Summary of main results}

This paper can be regarded as a "type C counterpart" of \cite{TV2} where a basis of integral solutions to the qKZ equations of type A was constructed. 
{}From the perspective of the Heisenberg XXZ spin chains, the present paper
deals with (diagonal) integrable reflecting boundary conditions as opposed to
quasi-periodic boundary conditions in \cite{TV2}. 

We consider the boundary qKZ equation \eqref{bqkz-int} for the tensor 
product of $N$ finite-dimensional representations of quantum affine 
$\mathfrak{sl}_2$ and/or Verma modules over quantum $\mathfrak{sl}_2$, 
for K-matrices which are diagonal in the weight basis.
For a fixed total weight space in the tensor product of the representations,
which is naturally labelled by a nonnegative integer $M$, 
we construct a basis for the associated
meromorphic solutions $\{\Psi_{\bm k}\}_{\bm k}$ of the boundary qKZ equations
over the field of $\tau \Z^N$-periodic 
meromorphic functions admitting an explicit integral representation
\begin{equation} \label{integrals-int}
\Psi_{\bm k}(\bm t) = \int_{C_{\bm k}(\bm t)}  w_{\bm k}(\bm x ; \bm t) \widetilde{\mathcal{B}}(\bm x;\bm t) \Omega \d^M \bm x
\end{equation}
for $\bm t$ deep enough in the asymptotic sector 
\[
\{\bm t \,\, | \,\, \Re(t_1) > \Re(t_2) > \cdots > \Re(t_N) >0\},
\]
where $\Re(t)$ is the real part of $t$. 
The index set consists of
$M$-tuples $\bm k = (k_1,\ldots,k_M)$ of integers satisfying $1 \leq k_1 \leq k_2 \leq \ldots \leq k_M \leq N$, in case of Verma modules. 
All the elements appearing in these integrals will be precisely defined in Section \ref{sec:integralsolutions}; we only emphasize some key points here.
The integration contours $C_{\bm k}(\bm t)$ are Cartesian products of line segments in the complex plane of length $\pi$ parallel to the imaginary axis, whose location depends on the variables $t_{k_1},\ldots,t_{k_M}$.
The weight functions $w_{\bm k}$ are similar to those used in \cite{RSV,RSV2}; more precisely, they are solutions to the same difference equations, but differ by exponential factors and quotients of Jacobi theta functions. 
The $\widetilde{\mathcal{B}}(\bm x;\bm t) \Omega$ are so-called off-shell Bethe vectors for reflecting boundary conditions as introduced in \cite{Sk}; they are, up to a scalar factor, those used in \cite{RSV,RSV2}.
The proof of completeness and of independence of solutions is done by a careful study of asymptotic behaviour of
solutions when $\Re(t_1) \gg \Re(t_2) \gg \cdots \gg \Re(t_N) \gg 0$.
It is this asymptotic analysis that relates the integral solutions of the present paper to power series solutions of boundary qKZ equations, as studied before in, e.g., \cite{vMS} and \cite{S}.

\subsection{Open problems}
It would be very interesting to understand the precise relation between the integral solutions and the solutions considered in \cite{JKKKM,JKKMW} in terms of matrix elements of vertex operators with respect to so-called boundary states,
as well as the Jackson integral solutions from \cite{RSV,RSV2}. See \cite{KN}
for some results in this direction for type A.
Another important problem is to verify the compatibility of these solutions $\Psi_{\bm k}(\bm t)$ with fusion (cf. \cite[Sec. 8]{RSV2}) and study their classical and rational limits.

Among other open problems it would be a natural continuation of present work to find integral solutions for the boundary qKZ equations \eqref{bqkz-int} for different R- and K-matrices.
It would be nice in all these cases to construct a basis of solutions, to understand their representation-theoretical meaning and to compare with special cases known from physics literature.

\subsection{Outline}
In Section \ref{sec:integrabledata} we review the construction of 
the boundary qKZ equations for tensor products of Verma modules over 
quantum $\mathfrak{sl}_2$. We restrict throughout the paper
to the special case that the associated solutions of the reflections equations
are diagonal with respect to the standard basis of the tensor product of Verma
modules. 
Section \ref{sec:integralsolutions} is the heart of the paper: here we introduce the building blocks necessary for our solutions, namely weight functions defined in terms of q-shifted factorials, the boundary Bethe vectors and the integration contours. 
We also define the integral solutions 
and state the main result of this paper, namely, that for a fixed total weight
the integral solutions form a basis
of the meromorphic solutions to \eqref{bqkz-int} over the field of 
$\tau \Z^N$-periodic meromorphic functions, taking values in the correponding (finite-dimensional) total weight space of 
the tensor product of Verma modules.

The proof of the main statement is spread out over Section \ref{sec:proof}.
In Section \ref{sec:findim} we consider the boundary qKZ equations for functions taking values in the tensor product of finite-dimensional modules over quantum affine $\mathfrak{sl}_2$.
The appendices \ref{app:bqKZasymptotics} and \ref{app:Bethevectors} provide technical statements underpinning the proof of the main result. 

\subsection{Acknowledgements}

The research of N.R. was supported by the Chern-Simons chair and by the NSF grant DMS-1201391; he also acknowledges support from QGM at Aarhus, where an important part of the work has been done.
B.V. is grateful to the University of California and the University of Amsterdam for hospitality; his work was supported by an NWO free competition grant and EPSRC grant EP/L000865/1.

\section{Representation theory and integrable data} \label{sec:integrabledata}

Here we will use conventions from \cite[Section 2.1]{RSV2} and refer to this publication for more detail and references.

\subsection{The quantum affine algebra $\mathcal{U}_\eta(\widehat{\mathfrak{sl}}_2)$ and the universal R-matrix} \label{sec:uniR}
Let $\eta\in\mathbb{C}$ such that $\e^\eta$ is not a root of unity and
write
\begin{equation*}
\left(\begin{matrix} a_{00} & a_{01}\\ a_{10} & a_{11}\end{matrix}\right)=
\left(\begin{matrix} 2 & -2\\ -2 & 2\end{matrix}\right).
\end{equation*}
We start with the unital 
Hopf algebra $\widehat{\mathcal{U}}_\eta:=
\mathcal{U}_\eta(\widehat{\mathfrak{sl}}_2)$ with deformation parameter 
$\e^\eta$ and generators $e_0$, $e_1$, $f_0$, $f_1$ and $\e^h$ 
($h\in\mathfrak{h}:=\mathbb{C}h_0\oplus\mathbb{C}h_1$) with
defining relations 
\begin{equation*}
\begin{split}
&\e^0=1,\qquad \e^{h+h^\prime}=\e^h \e^{h^\prime},\\
&\e^h e_i \e^{-h} = \e^{\alpha_i(h)} e_i, \qquad \e^h f_i \e^{-h} = \e^{-\alpha_i(h)} f_i,
\qquad 
\lbrack e_i,f_j \rbrack = \delta_{i,j} \frac{\e^{\eta h_i}-\e^{-\eta h_i}}
{\e^{\eta}-\e^{-\eta}}
\end{split}
\end{equation*}
for $h,h^\prime\in\mathfrak{h}$ and $i,j=0,1$,
and with the Serre relations
\[
\begin{array}{r}
e_i^3e_j-(\e^{2\eta}+1+\e^{-2\eta})e_i^2e_je_i+(\e^{2\eta}+1+\e^{-2\eta})e_ie_je_i^2-e_je_i^3 =0 \\[3pt]
f_i^3f_j-(\e^{2\eta}+1+\e^{-2\eta})f_i^2f_jf_i+(\e^{2\eta}+1+\e^{-2\eta})f_if_jf_i^2-f_jf_i^3  =0
\end{array}
\]
for $i\not=j$. Here the roots 
$\alpha_i\in\mathfrak{h}^*$ are defined by $\alpha_i(h_j):=a_{ij}$,

The coproduct $\Delta$ and the counit $\epsilon$ are determined by their action on generators:
\begin{equation*}
\begin{split}
\Delta(\e^h)&=\e^h\otimes \e^h,\\
\Delta(e_i)&=e_i\otimes 1+\e^{-\eta h_i}\otimes e_i,\\
\Delta(f_i)&=f_i\otimes \e^{\eta h_i}+1\otimes f_i
\end{split}
\end{equation*}
and
\begin{equation*}
\epsilon(\e^h)=1,\quad \epsilon(e_i)=0,\quad \epsilon(f_i)=0
\end{equation*}
for $h\in\mathfrak{h}$ and $i=0,1$.

Let $\widetilde{\mathcal{U}}_\eta$ be the extension of 
$\widehat{\mathcal{U}}_\eta$ with generators $\e^{xd}$ ($x\in\mathbb{C}$)
satisfying $\lbrack \e^h,\e^{xd}\rbrack=0$,
$\e^{xd}e_i=\e^{x\delta_{i,0}}e_i\e^{xd}$ and $\e^{xd}f_i=\e^{-x\delta_{i,0}}f_i\e^{xd}$.
Then $\widetilde{\mathcal{U}}_\eta$ turns into a quantized Kac-Moody algebra
in a natural way (see \cite[Chpt. 9]{EFK}). 
In particular, there is a universal R-matrix $R \in \widetilde{\mathcal{U}}_\eta \widehat{\otimes} \widetilde{\mathcal{U}}_\eta$
\cite{Dr,EFK,FR}, where $\widehat{\otimes}$ is suitably completed tensor product. 
The universal R-matrix has the form
\[
R =\exp(\eta (c\otimes d+d\otimes c)) \mathcal{R}
\]
where $c=h_0+h_1$ and $\mathcal{R} \in \widehat{\mathcal{U}}_\eta \widehat{\otimes} \widehat{\mathcal{U}}_\eta$ is the ``truncated universal R-matrix''.
We denote the opposite coproduct by $\Delta^{\rm op}$.
In the category of modules where $c$ acts by zero (zero-level representations), $\mathcal{R}$ satisfies all properties of the universal R-matrix:
\begin{gather*}
\mathcal{R} \Delta(a)=\Delta^{\rm op}(a) \mathcal{R}, \\
(\Delta \otimes \id)(\mathcal{R})= \mathcal{R}_{13} \mathcal{R}_{23} , \ (\id \otimes \Delta)(\mathcal{R})=\mathcal{R}_{13} \mathcal{R}_{12}
\end{gather*}

\subsection{Evaluation representations of quantum affine $\mathfrak{sl}_2$ }
$\widehat{\mathcal{U}}_\eta$ contains quantum $\mathfrak{sl}_2$  as the Hopf subalgebra $\mathcal{U}_\eta = \langle e_1, f_1, \e^{z  h_1} \rangle$.
For $\ell \in \C$, define a representation $\pi^\ell$ of $\mathcal{U}_\eta$ on $V^\ell = \bigoplus_{d=0}^\infty \C v^\ell_d$ by means of
\begin{align*}
\pi^\ell(\e^{z h_1})(v^{\ell}_d) &= \e^{2 (\ell-d)z} v^\ell_d, \\
\pi^\ell(e_1)(v^{\ell}_d) &= \frac{\sinh(d\eta)\sinh((2\ell+1-d)\eta)}{\sinh(\eta)^2} v^{\ell}_{d-1}, \qquad \text{with } v^{\ell}_{-1} := 0, \\
\pi^\ell(f_1)(v^{\ell}_d) &= v^{\ell}_{d+1},
\end{align*}
for $d \in \Z_{\geq 0}$.
The $\mathcal{U}_\eta$-module $(\pi^\ell,V^\ell)$ is the Verma module with highest weight $\ell$ and highest weight vector $v^\ell_0$.

Precisely if $\ell \in \frac{1}{2} \Z_{\geq 0}$, the maximal
$\mathcal{U}_\eta$-submodule $W^\ell$ is given by 
$\oplus_{d=2\ell+1}^\infty \C v^\ell_d$. We write 
$\widebar{V}^\ell := V^\ell / W^\ell$ for the 
irreducible finite-dimensional quotient.
Denote by $\pr^{\ell}$ the canonical map: $V^\ell \twoheadrightarrow \widebar{V}^\ell$.
The cosets
\[ \widebar{v}^\ell_d := \pr^{\ell}(v^\ell_d) = v^\ell_d + W^\ell \qquad \text{for } d \in \{0,1,\ldots,2\ell \}\]
form a weight basis of $\widebar{V}^\ell$.
If $\ell = \tfrac{1}{2}$ we write $\widebar{V} = \widebar{V}^{\frac{1}{2}}$ and $\widebar{v}_d = \widebar{v}^{\frac{1}{2}}_d$ for $d \in \{0,1\}$.

For each $x \in \C$ there exists a unique unit-preserving algebra homomorphism $\phi_x: \widehat{\mathcal{U}}_\eta \to \mathcal{U}_\eta$ satisfying
\begin{align*}
\phi_x(e_0) &= \e^{-x} f_1, & \phi_x(e_1) &= \e^{-x} e_1, \\
\phi_x(f_0) &= \e^x e_1, & \phi_x(f_1) &= \e^x f_1, \\
\phi_x(\e^{zh_0}) &= \e^{-zh_1}, & \phi_x(\e^{zh_1}) &= \e^{zh_1},
\end{align*} 
using which we define the evaluation representation $\pi^{\ell}_x = 
\pi^\ell \circ \phi_x: \widehat{\mathcal{U}}_\eta \to \End(V^\ell)$.

\begin{rema}
In \cite{RSV2} we wrote $M^\ell$ instead of $V^\ell$ and $\{ m^\ell_1, \ldots, m^\ell_{2\ell+1} \}$ instead of $\{ v^\ell_0, \ldots, v^\ell_{2\ell} \}$. Also, $V^\ell$ in \cite{RSV2} corresponds to the present notation $\widebar{V}^\ell$. 
The purpose of this is to have simplified notation for infinite-dimensional objects (by placing a bar over the symbol to denote the corresponding finite-dimensional objects) as we will focus on these in most of this paper. 
\end{rema}

\subsection{R-matrices}  \label{sec:Rmatrices}
We now pass from the universal R-matrix to its image under the evaluation representation. For details see e.g. \cite{FR}.
Let $\ell_1,\ell_2 \in \C$ and let $R^{\ell_1 \, \ell_2}(x-y) \in \End(V^{\ell_1}\otimes V^{\ell_2})$ be the scalar multiple of 
$(\pi^{\ell_1}_x \otimes \pi^{\ell_2}_y)(\mathcal{R})$ normalized by the condition
\[
R^{\ell_1 \, \ell_2}(x)(v^{\ell_1}_0 \otimes v^{\ell_2}_0)=v^{\ell_1}_0 \otimes v^{\ell_2}_0.
\]
Then $R^{\ell_1 \, \ell_2}(x)$ satisfies the Yang-Baxter equation \eqref{YBE}.
Furthermore, it satisfies unitarity, 
\begin{equation} \label{unitarity}
R^{\ell_1 \, \ell_2}(x)^{-1} = P^{\ell_2 \, \ell_1} R^{\ell_2 \, \ell_1}(-x) P^{\ell_1 \, \ell_2}
\end{equation}
as well as $P$-symmetry,
\begin{equation} \label{Psymm}
 P^{\ell_2 \, \ell_1} R^{\ell_2 \, \ell_1}(x)  P^{\ell_1 \ell_2} = R^{\ell_1 \ell_2}(x).
\end{equation}
See, for example, \cite[Lemma 2.1]{RSV2} for a proof of $P$-symmetry.
Another important property is the \emph{(higher-spin) ice rule} according to 
which $R^{\ell_1\ell_2}(x)$ preserves the total weight spaces 
\begin{equation}\label{icerule}
\begin{split}
(V^{\ell_1}\otimes V^{\ell_2})(M):=&\{v\in V^{\ell_1}\otimes V^{\ell_2} \,\, | \,\,
(\pi^{\ell_1}\otimes\pi^{\ell_2})(\Delta(\e^{h_1}))v=
\e^{2(\ell_1+\ell_2-M)}v\}\\
=&\textup{span}\{ v_{d_1}^{\ell_1}\otimes v_{d_2}^{\ell_2}\,\, | \,\,
d_1,d_2\in\mathbb{Z}_{\geq 0}\,\&\,\, d_1+d_2=M\}
\end{split}
\end{equation}
for all $M\in\mathbb{Z}_{\geq 0}$.

In the case where $\ell_1 \in \tfrac{1}{2}\Z_{\geq 0}$, there is a unique linear
operator $L^{\ell_1,\ell_2}(x)$ on 
$\widebar{V}^{\ell_1} \otimes V^{\ell_2}$ satisfying
\begin{equation} \label{Ldefn} (\pr^{\ell_1} \otimes \Id_{V^{\ell_2}}) R^{\ell_1\, \ell_2}(x) = L^{\ell_1\, \ell_2}(x)  (\pr^{\ell_1} \otimes \Id_{V^{\ell_2}}). \end{equation}
Furthermore, if $\ell_1,\ell_2 \in \tfrac{1}{2}\Z_{\geq 0}$, there is a unique 
linear operator $\widebar{R}^{\ell_1,\ell_2}(x)$ on 
$\widebar{V}^{\ell_1} \otimes \widebar{V}^{\ell_2}$ satisfying
\begin{equation} \label{Rdefn} (\pr^{\ell_1} \otimes \pr^{\ell_2}) R^{\ell_1\, \ell_2}(x) = \widebar{R}^{\ell_1\, \ell_2}(x)  (\pr^{\ell_1} \otimes  \pr^{\ell_2}). \end{equation}

In particular, we have the spin-half R-matrix $\widebar{R}(x):=\widebar{R}^{\frac{1}{2} \, \frac{1}{2}}(x)$ satisfying, for $d_1,d_2 \in \{0,1\}$,
\begin{equation} \label{Rspinhalf} \widebar{R}(x)(\widebar{v}_{d_1} \otimes \widebar{v}_{d_2} ) = \begin{cases} \widebar{v}_{d_1} \otimes \widebar{v}_{d_2}, & \text{if } d_1=d_2, \\ \frac{\sinh(x)}{\sinh(x+\eta)} \widebar{v}_{d_1} \otimes \widebar{v}_{d_2} +  \frac{\sinh(\eta)}{\sinh(x+\eta)} \widebar{v}_{d_2} \otimes \widebar{v}_{d_1}, & \text{if } d_1\ne d_2.  \end{cases} \end{equation}
The operators $\widebar{R}^{\ell_1 \, \ell_2}(x)$ can be recursively defined in terms of $\widebar{R}(x)$ through what is known as R-matrix fusion, see \cite[Sections 3.2 and 3.3]{RSV2} and references therein.

Given $\ell \in \C$ we also have the spin-half L-operators 
$L^\ell(x) := L^{\frac{1}{2}\, \ell}(x)$, 
meromorphically depending on $x \in \C$; from \eqref{YBE}, \eqref{Ldefn} and \eqref{Rdefn} it follows that they satisfy
\begin{equation}
\label{RLL} \widebar{R}_{00'}(x-y) L^\ell_0(x;\bm t) L^\ell_{0'}(y;\bm t) = L^\ell_{0'}(y;\bm t) L^\ell_0(x;\bm t) \widebar{R}_{00'}(x-y),
\end{equation}
an identity for operators in $\widebar{V} \otimes \widebar{V} \otimes V^\ell$, with the first copy of $\widebar{V}$ labelled 0 and the second labelled $0'$.
They also satisfy \emph{crossing symmetry}
\begin{equation} \label{Lcrossing}
L^\ell(x) = \frac{\sinh(x+(\tfrac{1}{2}-\ell)\eta)}{\sinh(x+(\tfrac{1}{2}+\ell)\eta)} \siy_0 L^\ell(-x-\eta)^{t_0} \siy_0,
\end{equation}
an identity for operators in $\widebar{V} \otimes V^\ell$, with $\widebar{V}$ labelled 0, $\siy = \left( \begin{smallmatrix} 0 & -\sqrt{-1} \\ \sqrt{-1} & 0 \end{smallmatrix} \right)$ and $t_0$ transposition relative to $\widebar{V}$, cf. \cite[Eq. (2.8)]{RSV2}.

\begin{rema}
We emphasize that the notation $R^{\ell_1,\ell_2}(x)$ in this paper corresponds to $\mathcal{R}^{\ell_1,\ell_2}(x)$ in \cite{RSV2} whereas the present notation $\widebar{R}^{\ell_1,\ell_2}(x)$ is called $R^{\ell_1,\ell_2}(x)$ in \cite{RSV2}.
\end{rema}

\subsection{K-matrices} \label{sec:Kmatrices}
We highlight here the key points from \cite[Section 4]{RSV2}, which the reader may consult for more details and references.
Let $R^{\ell_1\ell_2}(x)$ be as in Section \ref{sec:Rmatrices}. 
Because of P-symmetry \eqref{Psymm}, both reflection equations \eqref{REs} simplify to
\begin{equation} \label{RE}
\begin{aligned}
& R^{\ell_1 \ell_2}(x-y) K^{\ell_1}_{1}(x) R^{\ell_1 \ell_2}(x+y) K^{\ell_2}_{2}(y) = \\
& \qquad K^{\ell_2}_{2}(y) R^{\ell_1 \ell_2}(x+y) K^{\ell_1}_{1}(x) R^{\ell_2 \ell_1}(x-y).
\end{aligned}
\end{equation}
There is a one-parameter family $K^\ell(x;\xi) \in \End(V^{\ell})$ of solutions to \eqref{RE} which are diagonal in the weight basis.
They act on arbitrary weight vectors according to
\begin{equation}\label{Kdiagonal}
K^{\ell}(x;\xi) v_d^\ell = \left( \prod_{j=1}^{d}\frac{\sinh(\xi-x+(\ell+\frac{1}{2}-j)\eta)}{\sinh(\xi+x+(\ell+\frac{1}{2}-j)\eta)} \right) v_d^\ell,
\end{equation}
where $d \in \mathbb{Z}_{\geq 0}$.
In particular, they satisfy
\[
K^{\ell}(x;\xi)(v^\ell_0) = v^\ell_0
\]
and the unitarity condition
\[
K^{\ell}(x;\xi)^{-1} =  K^{\ell}(-x;\xi).
\]

For $\ell \in \tfrac{1}{2} \Z_{\geq 0}$, the natural projection $\pr^\ell$ to the quotient representation $\widebar{V}^\ell$ applied to $K^\ell(x;\xi)$ gives the corresponding solution $\widebar{K}^{\ell}(x;\xi) \in \End(\widebar{V}^\ell)$ to \eqref{RE}; furthermore we write $\widebar{K}(x;\xi) = \widebar{K}^{1/2}(x;\xi)$ as for the R-matrix.
The operators $\widebar{K}^{\ell}(x;\xi)$ can be recursively defined in terms of $\widebar{R}(x)$ and $\widebar{K}(x;\xi)$ through what is known as K-matrix (bulk-boundary) fusion, see e.g. \cite{MN}.

\begin{rema}
Similar to the R-matrices, the notation $K^\ell(x)$ in this paper corresponds to $\mathcal{K}^\ell(x)$ in \cite{RSV2} whereas the present notation $\widebar{K}^\ell(x)$ is called $K^\ell(x)$ in \cite{RSV2}.
\end{rema}

\subsection{Tensor products of evaluation representations}
Let $N \in \Z_{\geq 0}$ and fix $\bm \ell = (\ell_1,\ldots,\ell_N) \in \C^N$.
We will be considering linear operators on 
\[
V^{\bm \ell} := V^{\ell_1} \otimes \cdots \otimes V^{\ell_N}
\]
and write an arbitrary element of its natural basis as $v^{\bm \ell}_{\bm d}=v^{\ell_1}_{d_1} \otimes \cdots \otimes v^{\ell_N}_{d_N}$, where $\bm d= (d_1,\ldots,d_N) \in \Z_{\geq 0}^N$.
Taking into account the weight decomposition 
$V^{\ell}=\oplus_{d\geq 0} \C v^\ell_d$ with respect to the action of $\e^{h_1}$
we have the total weight decomposition
\[
V^{\bm \ell}=\bigoplus_{M =0}^\infty  V^{\bm \ell}(M)
\]
with
\begin{equation*}
\begin{split}
V^{\bm \ell}(M)&:=\bigoplus_{\bm d \in P_{N}(M)} \C v^{\bm \ell}_{\bm d},\\
P_{N}(M)&:=\{ \bm d = (d_1,\ldots,d_N) \in \Z^N_{\geq 0} \, | \, \sum_{s=1}^N d_s = M \}
\end{split}
\end{equation*}
(compare with 
\eqref{icerule}, which is the special case $N=2$). 

The tensor product basis $\{v_{\bm d}^{\bm\ell}\}_{\bm d\in P_N(M)}$
of the weight subspaces $V^{\bm \ell}(M)$ can be labelled in another natural way.
For $M,N \in \Z_{\geq 0}$, define
\begin{equation}\label{IMN}
I_{M,N} = \{ (k_1,\ldots,k_M) \in \{1,\ldots,N\}^M \, | \, 
k_1 \leq k_2 \leq \ldots \leq k_M \}.
\end{equation}
The following map is a bijection between $P_N(M)$ and $I_{M,N}$:
\[
\zeta_{M,N}: \; I_{M,N} \stackrel{\sim}{\to} P_{N}(M) : \;  \bm k \mapsto (n_{\bm k}(1),\ldots,n_{\bm k}(N)),
\]
where we have introduced the notation
\begin{equation}
\label{ndefn}
\hspace{-4pt} n_{\bm k}(r) := \# \{ i \in \{1,\ldots,M\} \, | \, k_i=r \} 
\end{equation}
for $\bm k \in \{1,\ldots,N\}^M$ and $r \in \{1,\ldots,N\}$. 
The fact that $\zeta_{M,N}$ is a bijection is clear if we write an $N$-tuple from $I_{M,N}$ as follows:
\[
(k_1,\dots, k_M)=(\underbrace{1,\dots, 1}_{n_{\bm k}(1)},\underbrace{2,\dots, 2}_{n_{\bm k}(2)}, \underbrace{3,\dots, 3}_{n_{\bm k}(3)}, \dots, \underbrace{N,\dots, N}_{n_{\bm k}(N)}).
\]
Note also that $\sum_{r=1}^N n_{\bm k}(r) = M$ for all $\bm k \in I_{M,N}$.
Using $\zeta_{M,N}$ we can parametrize the tensor product basis of
the weight subspace $V^{\bm \ell}(M)$ by elements in $I_{N,M}$:
\[
V^{\bm \ell}(M) = \bigoplus_{\bm k \in I_{M,N}} \C \Omega_{\bm k}  \]
where
\[
\Omega_{\bm k} := v_{\zeta_{M,N}(\bm k)}^{\bm\ell}=
v^{\ell_1}_{n_{\bm k}(1)}\otimes \dots \otimes v^{\ell_N}_{n_{\bm k}(N)}.
\]

\subsection{Boundary qKZ equations}
Fix $\bm \ell \in \C^N$ and $\tau,\xi_+,\xi_- \in \C$.
Given the R- and K-matrices defined in 
Sections \ref{sec:Rmatrices} and \ref{sec:Kmatrices}, 
the spin-$\bm \ell$ boundary qKZ equations are the following equations for meromorphic $V^{\bm\ell}$-valued functions $\Psi(\bm t)$ in $\bm t\in\C^N$:
\begin{equation} \label{bqKZ}
\Psi(\bm t+ \tau \bm e_r) = \mathcal{A}_r(\bm t) \Psi(\bm t), \qquad r\in\{1,\ldots,N\},
\end{equation}
with $\{\bm e_r\}_{r=1}^N$ the standard othonormal basis of $\C^N$ and the transport operators given by 
\begin{equation}
\label{bqKZtransportmatrix}
\begin{aligned}
\mathcal{A}_r(\bm t) &= R^{\ell_r \, \ell_{r+1}}_{r \, r+1}(t_r-t_{r+1}+\tau) \cdots R^{\ell_r \, \ell_N}_{r \, N}(t_r-t_N+\tau)  \\
 & \quad \times K^{\ell_r}_r(t_r+\tfrac{\tau}{2};\xi_+) R^{\ell_N \, \ell_r}_{N \, r}(t_N+t_r) \cdots R^{\ell_{r+1} \, \ell_r}_{r+1 \, r}(t_{r+1}+t_r) \\
 & \quad \times R^{\ell_{r-1} \, \ell_r}_{r-1 \, r}(t_{r-1}+t_r) \cdots R^{\ell_1 \, \ell_r}_{1 \, r}(t_1+t_r) K^{\ell_r}_r(t_r;\xi_-) \\
 & \quad \times \bigl( R^{\ell_1 \, \ell_r}_{1 \, r}(t_1-t_r) \bigr)^{-1} \cdots \bigl(R^{\ell_{r-1} \, \ell_r}_{r-1 \, r}(t_{r-1}-t_r) \bigr)^{-1} \hspace{-15pt}
\end{aligned}
\end{equation}
for $r \in \{1,\ldots,N\}$.
The Yang-Baxter equation \eqref{YBE} and the reflection equation \eqref{RE} yield the consistency conditions
\[ \mathcal{A}_r(\bm t+\tau \bm e_s)\mathcal{A}_s(\bm t)=\mathcal{A}_s(\bm t + \tau \bm e_r)\mathcal{A}_r(\bm t) \qquad \text{for } r,s \in \{1,\ldots,N\}. \]
When $\ell_s \in \frac{1}{2}\mathbb{Z}_{\geq 0}$ the equation projects to the corresponding quotient space $V^{\ell_1} \otimes \cdots \otimes V^{\ell_{s-1}} \otimes \widebar{V}^{\ell_s} \otimes V^{\ell_{s+1}} \otimes \cdots \otimes V^{\ell_N}$,
see \cite{RSV2}.\\

The finite-dimensional subspaces $V^{\bm\ell}(M)\subseteq V^{\bm\ell}$ ($M\geq 0$) are invariant subspaces for the transport operators $\mathcal{A}_r(\bm t)$ ($r=1,\ldots,N$) since the $R$-matrices satisfy the ice rule and since the $K$-matrices act diagonally with respect to the 
weight basis of the Verma module.
It follows that the meromorphic solutions of \eqref{bqKZ} are of the form 
$\Psi=\sum_{M\geq 0}\Psi_M$ with $\Psi_M$ a meromorphic 
$V^{\bm\ell}(M)$-valued solution of \eqref{bqKZ}. 
In the remainder of the paper we therefore focus on the construction of a basis of $V^{\bm\ell}(M)$-valued meromorphic
solutions of the boundary qKZ equations \eqref{bqKZ} for a fixed value of $M$.

We now first discuss the existence of power series solutions of the boundary
qKZ equations \eqref{bqKZ}, for fixed value of $M\in\mathbb{Z}_{\geq 0}$.
Suppose $\Re(\tau)<0$ and consider the sector
\[ 
\mathbb{A} := \{ (t_1,\ldots,t_N) \in \C^N \, | \, \Re(t_1) > \Re(t_2) > \ldots > \Re(t_N) > 0 \}. 
\]
We will write ${\bm t}\stackrel{\mathbb{A}}{\rightarrow}\infty$ when
\[
\Re(t_s-t_{s+1}) \to \infty \text{ for } s \in \{1,\ldots,N-1\} \qquad \text{and} \qquad \Re(t_N) \to \infty.
\]
Note that if ${\bm t}\stackrel{\mathbb{A}}{\rightarrow}\infty$ the real
parts of the arguments of all R- and K-matrices in 
\eqref{bqKZtransportmatrix} go to $+\infty$.

In order to study the asymptotics of the bqKZ equations \eqref{bqKZ} and their solutions as ${\bm t}\stackrel{\mathbb{A}}{\rightarrow}\infty$ more precisely, we rely on appendix \ref{app:bqKZasymptotics} and \cite[\S 9.6]{EFK}.

Let $Q_+\subseteq\mathbb{R}^N$ be the cone
\begin{equation}\label{Qplus}
Q_+:=\bigoplus_{i=1}^{N-1}\mathbb{Z}_{\geq 0}(\bm e_i-\bm e_{i+1})\oplus
\mathbb{Z}_{\geq 0}\bm e_N.
\end{equation}
For $\beta=(\beta_1,\ldots,\beta_N)\in\mathbb{Z}^N$ we write 
$\e^{(\beta,\bm t):}=\prod_{i=1}^N \e^{\beta_it_i}$. 
Let $r\in\{1,\ldots,N\}$. 
It follows from \cite[\S 9.6]{EFK} and the explicit form of the diagonal $K$-matrices that 
there exist $\mathcal{A}_{\alpha;r} \in \End(V^{\bm \ell}(M))$ for 
$\alpha \in Q^+$ such that
\[
\mathcal{A}_{r}(\bm t)=
\sum_{\alpha\in Q_+}\mathcal{A}_{\alpha;r} \e^{-(\alpha,\bm t)}
\]
as linear operator on $V^{\bm\ell}(M)$, with the power series converging normally on compact sets for $\bm t$ deep enough in the sector $\mathbb{A}$. 
Write $\bm 0 = (0,\ldots,0) \in \C^N$.
Lemma \ref{lem:Ainfinity} implies that $\mathcal{A}_{\bm 0;r}$ acts on $V^{\bm \ell}(M)$ by the formula
\begin{equation}\label{Ainfty}
\mathcal{A}_{\bm 0;r}\Omega_{\bm m}=\varphi_{\bm m;r}\Omega_{\bm m},\qquad
\bm m\in I_{M,N},
\end{equation}
where
\begin{equation} \label{varphidefn}
\varphi_{\bm m;r} := 
\Biggl( \prod_{i=1 \atop m_i=r}^M  
\e^{2(\eta - \xi_+ - \xi_-) +4(M-i-\sum_{s \geq r} \ell_s ) \eta}  \Biggr) 
\e^{-4 \ell_r \#\{i\in \{1,\ldots,M\}|m_i>r\}\eta }.
\end{equation}

Let $\Phi_{\bm k}$ be a nonzero meromorphic function on $\C^N$ satisfying the 
scalar difference equations
\begin{equation} \label{Phifunceqn}
\Phi_{\bm k}(\bm t + \tau \bm e_r) = \varphi_{\bm k;r} \Phi_{\bm k}(\bm t), \qquad 
r \in \{1 ,\ldots, N\}.
\end{equation}
Then it follows that 
\[
\Psi_{\bm k}^{\infty}(\bm t):=\Phi_{\bm k}(\bm t)\Omega_{\bm k}
\]
satisfies the asymptotic boundary qKZ equations
\[
\Psi_{\bm k}^{\infty}(\bm t+\tau \bm e_r)=
\mathcal{A}_{\bm 0;r}\Psi_{\bm k}^{\infty}(\bm t),
\qquad r=1,\ldots,N. 
\]
So it then makes sense to look for power series solutions
$\Psi_{\bm k}(\bm t)$ of \eqref{bqKZ} tending to $\Psi_{\bm k}^{\infty}(\bm t)$
as ${\bm t}\stackrel{\mathbb{A}}{\rightarrow}\infty$.

It is easy to construct an explicit solution $\Phi_{\bm k}$ of
\eqref{Phifunceqn} 
as a quotient of products of renormalized theta functions. We first need to
introduce some notations. 
We write $q=\e^\tau$ and we will assume $\Re(\tau) <0$, so that $0<|q|<1$.
The $q^2$-shifted factorial 
\[ (z;q^2)_\infty := \prod_{m \geq 0} (1-zq^{2m}) 
\]
is a holomorphic function of $z$ satisfying the property 
$(q^2z;q^2)_\infty = (1-z)^{-1} (z;q^2)_\infty$.
We will employ the notation
\[ (z_1,z_2;q^2)_\infty :=  (z_1;q^2)_\infty(z_2;q^2)_\infty.\]
The renormalized Jacobi theta function is the holomorphic function defined by
\[ \theta(z) := (z,q^2z^{-1};q^2)_\infty. \]
It satisfies the quasi-periodicity condition
$\theta(q^2 z) = -z^{-1} \theta(z)$. We will use in formulas the notation
$a(\pm x)$ to stand for $a(x)a(-x)$. 
For instance, $(\e^{\pm x};q^2)_\infty$ stands for $(\e^{x},\e^{-x};q^2)_\infty$.

An explicit nonzero meromorphic 
solution $\Phi_{\bm k}$ of \eqref{Phifunceqn} is now given by
\[
\Phi_{\bm k}(\bm t)=\prod_{r=1}^N\frac{\theta(\e^{2t_r})}
{\theta(\varphi_{\bm k;r} \e^{2t_r})}.
\]
Any other choice differs from
$\Phi_{\bm k}$ by a nonzero $\tau\mathbb{Z}^N$-periodic meromorphic function. 

If $\Psi$ is a $V^{\bm\ell}(M)$-valued meromorphic solution of the
boundary qKZ equations \eqref{bqKZ} and $\bm k\in I_{M,N}$ is fixed, then
\[
\widetilde{\Theta}_{\bm k}:=\Phi_{\bm k}^{-1}\Psi
\] 
is a $V^{\bm\ell}(M)$-valued meromorphic solution of the rescaled boundary
qKZ equations
\begin{equation}\label{rbqKZ}
\widetilde{\Psi}(\bm t+\tau \bm e_r) = 
\widetilde{\mathcal{A}}_{r}^{\bm k}(\bm t)\widetilde{\Psi}(\bm t), \qquad 
r\in\{1,\ldots,N\},
\end{equation}
with rescaled transport operators
$\widetilde{\mathcal{A}}_{r}^{\bm k}(\bm t):=
\varphi_{\bm k;r}^{-1}\mathcal{A}_r(\bm t)$. Now the leading coefficient 
$\widetilde{\mathcal{A}}_{\bm 0;r}^{\bm k}$ of
the power series expansion 
\[
\widetilde{\mathcal{A}}_r^{\bm k}(\bm t)=
\sum_{{\alpha}\in Q_+}\widetilde{\mathcal{A}}_{{\alpha};r}^{\bm k} \e^{-({\alpha},\bm t)}
\]
is acting on $V^{\bm\ell}(M)$ by  
\[
\widetilde{\mathcal{A}}_{\bm 0;r}^{\bm k}\Omega_{\bm m}=  \frac{\varphi_{\bm m;r}}{\varphi_{\bm k;r}} \Omega_{\bm m},\qquad
\bm m\in I_{M,N}.
\]
Then \cite[Appendix]{vMS} guarantees, for generic parameters, 
the existence and uniqueness of a $V^{\bm\ell}(M)$-valued meromorphic solution
$\widetilde{\Theta}_{\bm k}$ of the rescaled boundary qKZ equations \eqref{rbqKZ}
such that
\[
\widetilde{\Theta}_{\bm k}(\bm t)=\sum_{{\alpha}\in Q_+}
\widetilde{L}_{{\alpha}}^{\bm k} \e^{-({\alpha},\bm t)},\qquad 
\widetilde{L}_{{\alpha}}^{\bm k}\in V^{\bm\ell}(M)
\] 
for $\bm t$ deep enough in the sector $\mathbb{A}$, with the $V^{\bm\ell}(M)$-valued power series normally converging on compact sets
and with leading coefficient 
\[
\widetilde{L}_0^{\bm k}=\Omega_{\bm k}.
\]
Our main goal is to find an explicit integral expression of 
$\widetilde{\Theta}_{\bm k}$ deep enough in the sector $\mathbb{A}$. 

\section{Integral solutions of the boundary qKZ equations} \label{sec:integralsolutions}
We will exhibit $V^{\bm\ell}(M)$-valued solutions $\Psi_{\bm k}(\bm t)$ of 
\eqref{bqKZ} for $\bm k = (k_1,\ldots,k_M) \in I_{M,N}$
 admitting an integral representation of the form 
\[
\Psi_{\bm k}(\bm t)=
\int_{C_{\bm k}(\bm t)} w_{\bm k}(\bm x;\bm t) 
\widetilde{\mathcal{B}}(\bm x;\bm t) \Omega \d^M\bm x, 
\]
on some subsector $\widetilde{\mathbb{A}} \subset \mathbb{A}$.
We will show that, up to an explicit multiplicative constant,
$\Psi_{\bm k}$ equals $\Phi_{\bm k}\widetilde{\Theta}_{\bm k}$,
which provides the link with the power series solutions of the boundary
qKZ equations from the previous subsection.

The Bethe vectors $\widetilde{\mathcal{B}}(\bm x;\bm t) \Omega$ are 
elements of $V^{\bm \ell}(M)$ with trigonometric dependence on $\bm x\in\C^M$
and $\bm t\in\C^N$. They will be discussed in Subsection \ref{sec:Bethevectors}.
The scalar weight functions $w_{\bm k}(\bm x;\bm t)$ 
will be defined in Subsection \ref{sec:weightfunctions}.
We will specify the integration contour $C_{\bm k}(\bm t)$ in Subsection \ref{sec:contours} before stating the main theorem in Subsection \ref{sec:main}. 

\subsection{Bethe vectors} \label{sec:Bethevectors}
Like the qKZ transport operators $\mathcal{A}_r(\bm t)$, the objects $\widetilde{\mathcal{B}}(\bm x;\bm t)$ are linear operators on $V^{\bm \ell}$ constructed in terms of the R- and K-matrices introduced in Section \ref{sec:integrabledata}, but according to a different procedure, first conceived for quantum integrable systems with reflecting boundaries by Sklyanin \cite{Sk}.
We first introduce linear operators acting on $\widebar{V} \otimes V^{\bm \ell}$, where the solitary tensor factor $\widebar{V} \cong \C^2$ is called 
auxiliary space and is labelled 0.

Recall the L-operators $L^\ell(x)=L^{\frac{1}{2}\ell}(x) \in \End(\widebar{V} \otimes V^\ell)$ defined through \eqref{Ldefn}.
Fix $\bm \ell=(\ell_1,\ldots,\ell_N) \in \C^N$.
Define the (type A) monodromy matrix
\[ T(x;\bm t) = L^{\ell_1}_{0 \, 1}(x-t_1) \cdots L^{\ell_N}_{0 \, N}(x-t_N) \in \End(\widebar{V} \otimes V^{\bm \ell}). \]
From \eqref{RLL} it follows that the $T(x;\bm t)$ satisfy
\begin{equation}
\label{RTT} \widebar{R}_{00'}(x-y) T_0(x;\bm t) T_{0'}(y;\bm t) = T_{0'}(y;\bm t) T_0(x;\bm t) \widebar{R}_{00'}(x-y),
\end{equation}
an identity in $\End(\widebar{V} \otimes \widebar{V} \otimes V^{\bm \ell})$, with the first copy of $\widebar{V}$ labelled 0 and the second labelled $0'$.
Fix $\xi_- \in \C$ and define
\[ \mathcal{T}(x;\bm t) =  T(-x;\bm t)^{-1} \bigl( \widebar{K}(x;\xi_-) \otimes \Id_{V^{\bm \ell}} \bigr) T(x;\bm t)  \in \End(\widebar{V} \otimes V^{\bm \ell}). \]
Given that $T(x;\bm t)$ satisfies \eqref{RTT} and $\widebar{K}(x;\xi_-)$ satisfies the reflection equation in $\End(\widebar{V} \otimes \widebar{V})$, it can be straightforwardly verified that $ \mathcal{T}(x;\bm t)$ satisfies the reflection equation in $\End(\widebar{V} \otimes \widebar{V} \otimes V^{\bm \ell})$:
\begin{equation} \label{RTRT}
\begin{gathered}
\widebar{R}_{00'}(x-y) \mathcal{T}_0(x;\bm t) \widebar{R}_{00'}(x+y)  \mathcal{T}_{0'}(y;\bm t) =\\
 \mathcal{T}_{0'}(y;\bm t) \widebar{R}_{00'}(x+y) \mathcal{T}_0(x;\bm t) \widebar{R}_{00'}(x-y) .
\end{gathered}
\end{equation}
We introduce an operator $\mathcal{B}(x;\bm t) \in \End(V^{\bm \ell})$ by means of
\begin{equation} \label{bB0defn} \mathcal{T}(x;\bm t)  = \begin{pmatrix} \ast & \mathcal{B}(x;\bm t) \\ \ast & \ast \end{pmatrix}, \end{equation}
i.e. for all $u \in V^{\bm \ell}$ we have
\[ \mathcal{T}(x;\bm t)(\widebar{v}_1 \otimes u) = \widebar{v}_0 \otimes \mathcal{B}(x;\bm t)(u) + \widebar{v}_1 \otimes \text{(some element of $V^{\bm \ell}$)} .\]
In the rest of the paper we will use the shorthand notations
\[ \widetilde \xi_+ := \xi_+ - \tfrac{\eta}{2} - \tfrac{\tau}{2}, \qquad \widetilde \xi_- := \xi_- - \tfrac{\eta}{2}. \]
It is convenient, as will become apparent in Appendix \ref{app:bBethevectordecomposition}, to use a slightly modified version of $\mathcal{B}$, namely
\[ \widetilde{\mathcal{B}}(x;\bm t) :=  \biggl( \prod_{s=1}^N \frac{\sinh(t_s-x+\ell_s \eta)}{\sinh(t_s-x-\ell_s \eta)}\biggr)  \frac{\sinh(2x)}{\sinh(2x+\eta)} \frac{\sinh(\widetilde \xi_- - x)}{ \sinh(\eta)} \mathcal{B}( -x-\tfrac{\eta}{2};\bm t) . \]
Fix $M \in \Z_{\geq 0}$ and define, 
for $\bm x = (x_1,\ldots,x_M)$ and $\bm t=(t_1,\ldots,t_N)$,
\[ \widetilde{\mathcal{B}}(\bm x;\bm t) = \widetilde{\mathcal{B}}(x_1;\bm t) \cdots \widetilde{\mathcal{B}}(x_M;\bm t)\in
\textup{End}(V^{\bm\ell}) \]
(we do not specify the depth $M$ in the notation of $\widetilde{\mathcal{B}}$,
it will be clear from context).

The \emph{off-shell spin$-\bm \ell$ boundary} (or \emph{type C}) \emph{Bethe vectors} are the elements $\widetilde{\mathcal{B}}(\bm x;\bm t) \Omega \in 
V^{\bm \ell}(M)$, where $\Omega$ is the tensor product of highest weight vectors:
\[ \Omega := \Omega_{\emptyset} = v^{\ell_1}_0 \otimes \cdots \otimes v^{\ell_N}_0 \in V^{\bm \ell}.\]
One of the results of this paper is an explicit decomposition of the boundary Bethe vectors in terms of the tensor product basis of $V^{\bm \ell}(M)$.

\begin{prop} \label{thm:Bethevectordecomposition}
As meromorphic $V^{\bm\ell}(M)$-valued functions in $(\bm x,\bm t) \in \C^M\times
\C^N$, we have
\begin{equation} \label{bBdecomposition}
\widetilde{\mathcal{B}}(\bm x;\bm t) \Omega =  \sum_{\bm k \in I_{M,N} }  
\beta_{\bm k}(\bm x;\bm t) \Omega_{\bm k}
\end{equation}
where
\begin{align*}
\hspace{-9pt} \beta_{\bm k}(\bm x;\bm t)
&= \e^{\sum_i \bigl(\tfrac{n_{\bm k}(k_i)}{2} - \ell_{k_i}\bigr) \eta} 
\sum_{\bm m \in S_M\bm k}  \sum_{\bm \epsilon \in \{\pm\}^M} \Biggl( \prod_i  \frac{\epsilon_i  \sinh(\epsilon_i x_i  - \widetilde \xi_-)}{\sinh(t_{m_i}+\epsilon_i x_i  - \ell_{m_i} \eta)}  \\
& \hspace{25pt}  \times  \biggl( \prod_{s>m_i} \frac{\sinh(t_s +\epsilon_i x_i +  \ell_s \eta)}{\sinh(t_s + \epsilon_i x_i  - \ell_s \eta)} \biggr) \biggl( \prod_s \frac{\sinh(t_s - \epsilon_i x_i + \ell_s \eta)}{\sinh(t_s - \epsilon_i x_i -  \ell_s \eta)} \biggr)  \Biggr) \hspace{-10pt} \\
& \hspace{15pt}   \times \biggl( \prod_{i,j \atop i<j} \hspace{-2pt} \frac{\sinh(\epsilon_i x_i + \epsilon_j x_j + \eta)}{\sinh(\epsilon_i x_i + \epsilon_j x_j)} \biggr)  \biggl( \hspace{-2pt} \prod_{i,j \atop m_i<m_j} \hspace{-4pt} \frac{\sinh( \epsilon_i x_i - \epsilon_j x_j - \eta))}{\sinh( \epsilon_i x_i - \epsilon_j x_j)} \biggr).
\end{align*}
\end{prop}
Essentially, the result owes to the higher-spin ice rule of the $R$-operators and the diagonality of $\widebar{K}(x;\xi_-)$.
The proof of the Theorem is given in appendix \ref{app:Bethevectors}.

Note that Proposition 
\ref{thm:Bethevectordecomposition} implies that the off-shell boundary Bethe vectors 
$\widetilde{\mathcal{B}}(\bm x;\bm t) \Omega$
 are $\pi \sqrt{-1}$-periodic in each $x_i$.

\subsection{Weight functions} \label{sec:weightfunctions}

Fix $\tau, \eta, \xi_+,\xi_- \in \C$ and $\bm \ell \in \C^N$. Suppose that
$\Re(\tau)<0$ and set $q=\e^\tau$.
For $\bm t \in \C^N$, introduce the single-variable meromorphic $\pi\sqrt{-1}$-periodic functions $F(\cdot;\bm t)$, $g$, $h$ as follows:
\begin{equation}
\label{Fghdefn} \begin{aligned}
F(x;\bm t) &:= \prod_{s=1}^N \frac{\bigl(\e^{-2(t_s \pm x - \ell_s \eta)}
;q^2\bigr)_{\infty}}{\bigl(\e^{-2(t_s \pm x + \ell _s \eta)};q^2\bigr)_{\infty}}, \\
g(x) &:= \frac{(q^2 \e^{2(\widetilde \xi_+ -x)},q^2 \e^{2(\widetilde \xi_- - x)};q^2)_\infty}{(\e^{2(- \widetilde \xi_+-x)},\e^{2(-\widetilde \xi_- - x)};q^2)_\infty},\\
h(x) &:= (1-\e^{-2x}) \frac{(q^2 \e^{-2(x+\eta)};q^2)_\infty}{(\e^{-2(x-\eta)};q^2)_\infty}.
\end{aligned}
\end{equation}
For $i \in \{1,\ldots,M\}$ and $\bm k  \in I_{M,N}$ define the single-variable meromorphic $\pi\sqrt{-1}$-periodic function $u_{\bm k;i}(\cdot;\bm t)$ by
\begin{equation}
\label{udefn}
u_{\bm k;i}(x;\bm t) := \e^{-t_{k_i}} \theta\bigl(\e^{2(x-t_{k_i}+\psi_{\bm k;i})}\bigr) \frac{  \prod_{s>k_i} \theta(\e^{2(x-t_s- \ell_s \eta)})}{ \prod_{s\geq k_i} \theta(\e^{2(x-t_s+\ell_s \eta)})}
\end{equation}
where
\[ \psi_{\bm k;i} = \widetilde \xi_+ + \widetilde \xi_- + \tau + \ell_{k_i}\eta +  2\bigl( \sum_{s > k_i} \ell_s - M + i\bigr) \eta. \]
We use these functions as building blocks to define a 
meromorphic weight function $w_{\bm k}$ on 
$\C^M \times \C^N$ by
\begin{equation}
\label{wdefn} w_{\bm k}(\bm x;\bm t) := \Phi_{\bm k}(\bm t) \biggl( \prod_{i=1}^M F(x_i ; \bm t) g(x_i) u_{\bm k;i}(x_i;\bm t) \biggr)  \prod_{1\leq i<j\leq M} h(x_i \pm x_j) \hspace{10mm}
\end{equation}
By writing out the theta-functions as products of 
$q^2$-shifted factorials, cancellation with $q^2$-shifted factorials coming from
$F$ takes place. It leads to the expression
\begin{equation}\label{wformula}
\begin{aligned} 
\hspace{-8pt} w_{\bm k}(\bm x;\bm t)&=\Phi_{\bm k}(\bm t)  \Biggl\{ \prod_{i=1}^M\frac{ \prod_{s<k_i} ( \e^{2(x_i - t_s + \ell_s \eta)};q^2)_\infty}{ \prod_{s \leq k_i} ( \e^{2(x_i - t_s -\ell_s \eta)};q^2)_\infty}  \e^{-t_{k_i}}  \theta\bigl(\e^{2(x_i-t_{k_i}+\psi_{\bm k;i})}\bigr)   \\
& \hspace{18mm} \times  \frac{\prod_{s>k_i} (q^2 \e^{2(t_s-x_i+\ell_s \eta)};q^2)_\infty}{\prod_{s \geq k_i} (q^2 \e^{2(t_s-x_i-\ell_s \eta)};q^2)_\infty} \bigg( \! \prod_s \! \frac{( \e^{-2(t_s+x_i-\ell_s \eta)};q^2)_\infty}{( \e^{-2(t_s+x_i+\ell_s \eta)};q^2)_\infty} \! \bigg) \hspace{-8pt} \\
& \hspace{18mm} \times  \frac{(q^2 \e^{2(\widetilde \xi_+ -x_i)},q^2 \e^{2(\widetilde \xi_- -x_i)};q^2)_\infty}{( \e^{2(- \widetilde \xi_+-x_i)},\e^{2(-\widetilde \xi_- -x_i)};q^2)_\infty} \Biggr\} \\
& \qquad \qquad \times \Biggl\{ \prod_{i,j =1 \atop i <j}^M (1-\e^{-2( x_i \pm x_j)}) \frac{(q^2 \e^{-2(x_i \pm x_j + \eta)};q^2)_\infty}{(\e^{-2(x_i \pm x_j - \eta)};q^2)_\infty}\Biggr\}.
\end{aligned} \end{equation}
The poles of $w_{\bm k}(\bm x; \bm t)$ in $x_i$ 
are unilateral sequences whose real parts tend to either 
$\infty$ or $-\infty$ in steps of size $-\Re(\tau)$.
Also, $w_{\bm k}(\bm x;\bm t)$ is $\pi \sqrt{-1}$-periodic in each $x_i$.

\subsection{Integration contours} \label{sec:contours}

We fix $\tau\in\mathbb{C}$ with $\Re(\tau)<0$ and set $q=\e^{\tau}$.
We furthermore fix $N\in\mathbb{Z}_{>0}$ and $M\in\mathbb{Z}_{\geq 0}$.
Write $\bm\ell=(\ell_1,\ldots,\ell_N)$ for the $N$-tuple of highest
weights.

\begin{defi}
For $\bm k\in I_{M,N}$ we write
$\mathcal{D}_{M,N}^{\bm k}$ for the set of parameters $(\bm\ell,\eta)\in\C^N\times\C$ satisfying 
\[
\Re(\ell_r\eta) > \max\biggl( 0,\frac{n_{\bm k}(r)-1}{2}\Re(\eta) \biggr)
\qquad\forall\, r\in\{1,\ldots,N\}.
\]
Furthermore, set
\begin{equation}\label{parameterdomain}
\mathcal{D}_{M,N}:=\bigcap_{\bm k\in I_{M,N}}\mathcal{D}_{M,N}^{\bm k}.
\end{equation}
\end{defi}

For $\bm k\in I_{M,N}$ we write
\begin{equation}\label{ibmk}
i_{\bm k}(m;r):=\sum_{s<r}n_{\bm k}(s)+m
\end{equation}
for $r\in\{1,\ldots,N\}$ and $m\in\{1,\ldots,n_{\bm k}(r)\}$,
so that $k_i=r$ if and only if $i=i(m;r)$ for some 
$1\leq m\leq n_{\bm k}(r)$. We now define for $\bm k\in I_{M,N}$ 
and $(\ell,\eta)\in\mathcal{D}_{M,N}^{\bm k}$ the set
$\Gamma_{M,N}^{\bm k}$ of base points of the integration cycle
as the set of $M$-tuples
$\bm\gamma=(\gamma_{1},\ldots,\gamma_{M})\in\C^M$
satisfying
\begin{equation*}
\begin{split}
&-\Re(\ell_{k_i} \eta)<\Re(\gamma_{i})<\Re(\ell_{k_i} \eta),\\
& \Re(\gamma_{i(s+1;r)}) \leq \Re(\gamma_{i(s;r)}),\\
& \Re(\gamma_{i(s+1;r)})+\Re(\eta)<\Re(\gamma_{i(s;r)})
\end{split}
\end{equation*}
for $1\leq i\leq M$, $1\leq r\leq N$ and $1\leq s<n_{\bm k}(r)$.
Note that $\Gamma_{M,N}^{\bm k}$
is nonempty and path-connected if 
$(\bm\ell,\eta)\in\mathcal{D}_{M,N}^{\bm k}$.

Let $\bm k\in I_{M,N}$, $(\bm\ell,\eta)\in\mathcal{D}_{M,N}^{\bm k}$ and 
$\bm\gamma\in\Gamma_{M,N}^{\bm k}$. The integration cycles that will
be used in the definition of the integral solution of the boundary qKZ
equations \eqref{bqKZ} are of the form
\[ C_{\bm k}^{\bm\gamma}(\bm t) := \bigl(t_{k_1}+\gamma_{1}+
\sqrt{-1}[0,\pi]\bigr) \times \cdots \times \bigl(t_{k_M}+\gamma_{M}+
\sqrt{-1}[0,\pi]\bigr)
\]
for $\bm t=(t_1,\ldots,t_N)\in\C^N$. 


\subsection{Main result} \label{sec:main}

The set $\{\mathbf{e}_i-\mathbf{e}_{i+1}\}_{i=1}^{N-1}\cup\{\mathbf{e}_N\}$
is a choice of simple roots of the standard realization of the root system of type $\textup{B}_N$ in $\mathbb{R}^N$. Let $Q_+$ be the cone in $\mathbb{Z}^N$ generated by these simple roots, see \eqref{Qplus}. Denote
\[ \Lambda_+ := \max_{s \in \{1,\ldots,N\}} \Re(\ell_s \eta). \]

The main result of the paper can now be stated as follows.
\begin{thm}\label{thm:main}
Fix $N\in\mathbb{Z}_{>0}$ and $M\in\mathbb{Z}_{\geq 0}$. Set
$q=\e^{\tau}$ with $\Re(\tau)<0$. 
\begin{enumerate}
\item[{\bf a.}]  Let $\bm k \in I_{M,N}$ and 
$(\bm\ell,\eta)\in\mathcal{D}_{M,N}^{\bm k}$. For 
$\bm \gamma \in \Gamma_{M,N}^{\bm k}$ the integral
\begin{equation} \label{integrals} \Theta_{\bm k}(\bm t):= 
\int_{C_{\bm k}^{\bm\gamma}(\bm t)} \frac{w_{\bm k}(\bm x ; \bm t)}{\Phi_{\bm k}(\bm t)} \widetilde{\mathcal{B}}(\bm x;\bm t) \Omega \d^M \bm x \end{equation}
defines a $V^{\bm\ell}(M)$-valued holomorphic function in $\bm t\in
\widetilde{\mathbb{A}}$, where
\[
\begin{aligned}
\widetilde{\mathbb{A}} := 
\Bigl\{ \bm t \in \mathbb{A} \, \Big| \, & \Re(t_s-t_{s+1}) > 2\Lambda_+ + 
\max(\Re(\eta),0) \text{ for } 1 \leq s< N,\\ 
& \quad \Re(t_N)> \Lambda_+ + \max\bigl(\Re(\tfrac{\eta}{2}),0,\Re( -\widetilde \xi_+),\Re(-\widetilde \xi_-)\bigr) \Bigr\}. 
\end{aligned}
\]
The integral $\Theta_{\bm k}(\bm t)$
does not depend on the choice of $\bm \gamma \in \Gamma_{M,N}^{\bm k}$.
\item[{\bf b.}] Let $\bm k\in I_{M,N}$ and 
$(\bm\ell,\eta)\in\mathcal{D}_{M,N}^{\bm k}$. There exists a unique 
$V^{\bm\ell}(M)$-valued meromorphic solution $\Psi_{\bm k}$ of the boundary qKZ equations \eqref{bqKZ} on $\C^N$ such that, 
on the sector $\widetilde{\mathbb{A}}$,
\begin{equation} \label{scaledintegrals} \Psi_{\bm k}(\bm t)= \Phi_{\bm k}(\bm t)  \Theta_{\bm k}(\bm t)  =  \int_{C^{\bm \gamma}_{\bm k}(\bm t)} w_{\bm k}(\bm x ; \bm t)  \widetilde{\mathcal{B}}(\bm x;\bm t) \Omega \d^M \bm x. \end{equation}
\item[{\bf c.}] Let $\bm k\in I_{M,N}$ and $(\bm\ell,\eta)\in
\mathcal{D}_{M,N}^{\bm k}$. The $V^{\bm\ell}(M)$-valued
integral $\Theta_{\bm k}(\bm t)$ has a $V^{\bm\ell}(M)$-valued
series expansion for $\bm t\in\widetilde{\mathbb{A}}$ of the form
\begin{equation}\label{integralasymptotics}
\Theta_{\bm k}(\bm t)=\sum_{{\alpha}\in Q_+}L_{\alpha}^{\bm k} \e^{-({\alpha},\bm t)},\qquad L_{\alpha}^{\bm k}\in V^{\bm\ell}(M),
\end{equation}
with the series converging normally for $\bm t$ in compact subsets of $\widetilde{\mathbb{A}}$. 
The leading coefficient is given by
$L_{\bm 0}^{\bm k}=\nu_{\bm k}\Omega_{\bm k}$, where
\begin{gather} 
\label{nuproduct}
\begin{aligned}
\nu_{\bm k}  &:= \bigl( \pi \sqrt{-1} \e^{\xi_-} \bigr)^M \biggl( \prod_{1 \leq r < s \leq N} \e^{2 \ell_r n_{\bm k}(s) \eta} \biggr) \\
& \qquad \times  \prod_{r=1}^N \prod_{m=1}^{n_{\bm k}(r)} \frac{(q^2 \e^{-2m\eta}, q\e^{2((m-1-\ell_r)\eta \pm \omega_{\bm k;r})};q^2)_\infty}{(q^2,q^2\e^{-2\eta},\e^{2(m-1-2\ell_r)\eta};q^2)_\infty}, \hspace{-3mm}
\end{aligned} \\
\label{omega} \omega_{\bm k;r} := \xi_+ + \xi_- + \tau + (\ell_r - n_{\bm k}(r))\eta + 2 \sum_{s>r} (\ell_s - n_{\bm k}(s) ) \eta. 
\end{gather}
\item[{\bf d.}] Let $(\bm\ell,\eta)\in\mathcal{D}_{M,N}$. Then 
$\{ \Psi_{\bm k} \, | \, \bm k \in I_{M,N}\}$ is a linear basis of the space 
of $V^{\bm \ell}(M)$-valued meromorphic solutions of the boundary 
qKZ equations \eqref{bqKZ} over the field of $\tau \Z^N$-periodic 
meromorphic functions.
\end{enumerate}
\end{thm}

\begin{rema}
We may write 
\[  \Psi_{\bm k}(\bm t) =  \sum_{\bm m \in I_{M,N} } 
\biggl( \int_{C_{\bm k}^{\bm\gamma}(\bm t)}  w_{\bm k}(\bm x;\bm t) \beta_{\bm m}(\bm x;\bm t)  \d^M \bm x  \biggr) \Omega_{\bm m} \] 
for $\bm t\in\widetilde{\mathbb{A}}$ with the coefficients $\beta_{\bm m}$ 
as given in Theorem \ref{thm:Bethevectordecomposition}. 
It allows for a direct comparison with the integral solutions to type A 
qKZ in \cite{TV2}. 
The coefficients $\beta_{\bm m}(\bm x;\bm t)$ of the Bethe vectors 
are the boundary analogons of the ``trigonometric weight functions'' appearing in \cite{TV2}, whereas $w_{\bm k}(\bm x;\bm t)$ corresponds to the product of the ``short phase function'' and the ``elliptic weight function'' and $\Phi_{\bm k}(\bm t)$ is the direct counterpart of the ``adjusting factor'' of the elliptic weight function. 
More precisely, our integrand is the direct analogon of the integrand $F(t)$ considered in \cite[p.43, Proof of Theorem 6.6]{TV2}.
Here we add that in \cite{TV2} the elliptic weight function, like the trigonometric weight function, is defined as an orbit sum.
However, owing to invariance properties under an action of the symmetric group, its terms all contribute the same to the integral. 
In \cite{KN}, this type A elliptic weight function was re-defined as a single term in order to make a connection with solutions to the qKZ equations in terms of formulae derived from free-field realizations of intertwiners of quantum affine $\mathfrak{sl}_2$.
\end{rema}

The version of Theorem \ref{thm:main} for solutions of the boundary
qKZ equations \eqref{bqKZ} taking values in the tensor product of finite
dimensional modules over quantum $\mathfrak{sl}_2$ is discussed in Section 
\ref{sec:findim}.


\section{Proof of the main results} \label{sec:proof}

Here we will prove the various statements made in Thm. \ref{thm:main}. 
We start by a helpful lemma listing the poles of the integrand of $\Theta_{\bm k}(\bm t)$.

\begin{lem} \label{lem:poles}
Let $j \in \{1,\ldots,M\}$. 
Fix $x_i \in  t_{k_i} + \gamma_{i} + \sqrt{-1} [0,\pi] $ for $i \ne j$.
The poles of the integrand $w_{\bm k}(\bm x;\bm t) \widetilde{\mathcal{B}}(\bm x;\bm t) \Omega$ as a function of $x_j$ are contained in 
\[ \Bigl( (P^+_{\bm k;j}(\bm t) - \tau \Z_{\geq 0}) \cup (P^-_{\bm k;j}(\bm t) + \tau \Z_{\geq 0})  \Bigr)+\pi \sqrt{-1} \Z, \]
where
\begin{align*}
P^+_{\bm k;j}(\bm t) &= { \{   t_s + \ell_s \eta \}_{  s \leq k_j }} \cup \{ x_i - \eta  \}_{i<j} \\
P^-_{\bm k;j}(\bm t) &= \{ t_s - \ell_s \eta\}_{s \geq k_j } \cup \{ -t_s -\ell_s \eta \}_{s} \cup \{ -\widetilde \xi_+, -\widetilde \xi_- \} \cup \\
& \qquad  \cup \{ x_i + \eta \}_{i>j} \cup \{ - x_i + \eta \}_{i \ne j }.
\end{align*}
\end{lem}

\begin{proof}
Since $\Re(\tau)<0$, from \eqref{wformula} we see that the sequences of poles of $x_j \mapsto w_{\bm k}(\bm x;\bm t)$ whose real parts tend to $+\infty$ are given by
\[  \Bigl( \bigl\{ t_s + \ell_s \eta \, \big| \, 1 \leq s \leq k_j  \bigr\} \cup \bigl\{  x_i - \eta  \, \big| \, i  < j  \bigr\} \Bigr) - \tau \Z_{\geq 0} + \pi \sqrt{-1} \Z \]
whereas the sequences of poles whose real parts tend to $-\infty$ are given by
\begin{align*}
& \Bigl( \bigl\{ t_s - \ell_s \eta + \tau \, \big| \, k_j  \leq s \leq N \bigr\} \cup \bigl\{ -t_s - \ell_s \eta \, \big| \, 1 \leq s \leq N \bigr\} \cup \{ -\widetilde \xi_+, -\widetilde \xi_- \} \cup \\
& \hspace{32mm} \cup \bigl\{ x_i+\eta \, \big| \, i>j \bigr\} \cup \bigl\{  -x_i + \eta  \, \big| \, i  \ne  j  \bigr\} \Bigr) + \tau \Z_{\geq 0} + \pi \sqrt{-1} \Z. 
\end{align*}
From Proposition \ref{thm:Bethevectordecomposition} we see that the poles of $x_j \mapsto \beta_{\bm k'}(\bm x;\bm t)$ are contained in
\[   \Bigl( \bigl\{ \pm (t_s - \ell_s \eta) \, \big| \, 1 \leq s \leq N \bigr\} \cup \bigl\{ \pm x_i \, \big| \, i  \ne j  \bigr\} \Bigr) + \pi \sqrt{-1} \Z. \]
Using \eqref{wformula} again, we see that of these poles all but those of the form $t_s-\ell_s \eta$ for $1 \leq s \leq k_j$ are cancelled by zeros of $w_{\bm k}(\bm x; \bm t)$.
We obtain the desired statement.
\end{proof}

In the following four subsections we prove the four parts of Theorem 
\ref{thm:main}.

\subsection{Holomorphicity of the $\Theta_{\bm k}$}
Here and in Section \ref{sec:asymptoticscompleteness} we will use the shorthand notation
\[ \bm t_{\bm k} := (t_{k_1},\ldots,t_{k_M}) \in \C^M \]
for $\bm t = (t_1,\ldots,t_N) \in \C^N$ and $\bm k = (k_1,\ldots,k_M) \in I_{M,N}$.
We may substitute $x_i = y_i + t_{k_i}$ for $i \in \{1,\ldots,M\}$ in the defining formula \eqref{integrals}, 
so that 
\begin{equation} \label{integraltransformed}  \Theta_{\bm k}(\bm t) =   \int_{C^{\bm \gamma}_{\bm k}(\bm 0)} \frac{w_{\bm k}(\bm t_{\bm k} + \bm y; \bm t)}{\Phi_{\bm k}(\bm t)} \widetilde{\mathcal B}(\bm t_{\bm k} + \bm y;\bm t)  \Omega \d^M \bm y. \end{equation}
Note that the integration over $y_i$ is over the line segment 
$\gamma_{i} + \sqrt{-1} [0,\pi]$.

From Lemma \ref{lem:poles} we see that, for any $j \in \{1,\ldots,M\}$, the poles of the integrand as a function of $y_j$ are avoided if the real part of every element of $P^+_{\bm k;j}(\bm t)$ (with $x_i = y_i + t_{k_i}$ for $i \ne j$) exceeds $\Re(t_{k_j}+y_j)$, and the real part of every element of $P^-_{\bm k;j}(\bm t)$ (with $x_i = y_i + t_{k_i}$ for $i \ne j$) is less than $\Re(t_{k_j}+y_j)$. 
Note that $\Re(y_i)=\Re(\gamma_{i})$ for all $i \in \{1,\ldots,N\}$.
This yields the inequalities
\[
\begin{array}{r@{\hspace{3pt}}ll} 
|\Re(\gamma_{j})| &< \Re(\ell_{k_j} \eta), & \\
\Re(\gamma_{j} - \gamma_{j}) &> \Re(\eta), & \text{for } i<j \text{ and } k_i=k_j,  \\
\Re(\gamma_{j} - \gamma_{i}) &> \Re(\eta), & \text{for } i>j \text{ and } k_i=k_j, \\
\Re(t_s - t_{k_j}) &> \Re(\gamma_{j} - \ell_s \eta), & \text{for } 1 \leq s < k_j, \\
\Re(t_{k_j} - t_s) &> \Re(-\gamma_{j} - \ell_s \eta), & \text{for } k_j < s \leq N, \\
\Re(t_s+t_{k_j}) &> \Re(-\gamma_{j} - \ell_s \eta), & \text{for } 1 \leq s \leq N, \\
\Re(t_{k_i} - t_{k_j}) &> \Re(\gamma_{j} - \gamma_{i} + \eta),  & \text{for } i<j \text{ and } k_i<k_j, \\
\Re(t_{k_j} - t_{k_i}) &> \Re(\gamma_{i}- \gamma_{j} + \eta),  & \text{for } i>j \text{ and } k_i>k_j, \\
\Re(t_{k_i}+t_{k_j}) &> \Re(-\gamma_{i}- \gamma_{j} + \eta), & \text{for } i \ne j, \\
\Re(t_{k_j}) &> \Re(-\widetilde \xi_+ - \gamma_{j}) , & \\ 
\Re(t_{k_j}) &> \Re(-\widetilde \xi_- - \gamma_{j}). &
\end{array}
\]
The inequalities independent of $\bm t$ are a simple consequence of the condition $\bm \gamma \in \Gamma_{M,N}^{\bm k}$. 
For the inequalities involving sums and differences of the $t_s$ we also need $\Re(t_r-t_{r+1})>2\Lambda_+ + \max(\Re(\eta),0)$ for $1 \leq r< N$, 
$\Re(t_N)>\Lambda_+ + \max(\Re(\tfrac{\eta}{2}),0)$ and $\Re(t_N)>0$. 
The final two inequalities rely on $\Re(t_N) > \Lambda_+ + \max(\Re(-\widetilde \xi_+),\Re(-\widetilde \xi_-))$.

Since the integration in \eqref{integraltransformed} is over a compact set which is independent of $\bm t$ 
the integrals define $V^{\bm\ell}(M)$-valued 
holomorphic functions $\Theta_{\bm k}$ on $\widetilde{\mathbb{A}}$.

By virtue of Cauchy's integral theorem, these functions $\Theta_{\bm k}$ do not depend on $\bm \gamma \in \Gamma_{M,N}^{\bm k}$ 
since $\Gamma_{M,N}^{\bm k}$ is path-connected and the separation of poles by
the contours is unaltered for different choices of $\gamma$. 
We obtain part {\bf a} of Theorem \ref{thm:main}.

\subsection{Integral solutions of the boundary qKZ equations} \label{sec:bqKZverification}

It is convenient to define $F_{\bm k}$ to be the $V^{\bm \ell}(M)$-valued meromorphic function $\Phi_{\bm k} \Theta_{\bm k}$ on $\widetilde{\mathbb{A}}$.
Consider the subsector
\begin{align*} 
\widetilde{\mathbb{A}}_\tau := 
\Bigl\{ \bm t \in \mathbb{A} \, \Big| \, & \Re(t_s-t_{s+1}) > 2\Lambda_+ + \max( \Re(\eta) , 0 ) - \Re(\tau)
\text{ for } 1 \leq s< N,\\ 
& \qquad \Re(t_N)> \Lambda_+ + \max\bigl( \Re(\tfrac{\eta}{2}), 0 ,\Re(-\widetilde \xi_+),\Re(-\widetilde \xi_-) \bigr) - \Re(\tau) \Bigr\} 
\end{align*}
of $\widetilde{\mathbb{A}}$.
Note that if 
$\bm t \in \widetilde{\mathbb{A}}_\tau \subset \widetilde{\mathbb{A}}$ then
$\bm t + \tau \bm e_r \in \widetilde{\mathbb{A}}$ for all $r \in \{1,\ldots,N\}$ so that both sides of the boundary qKZ equations \eqref{bqKZ} with $\Psi = F_{\bm k}$ as given by \eqref{scaledintegrals} are well-defined on 
$\widetilde{\mathbb{A}}_\tau$. Here we will show that they are equal
on $\widetilde{\mathbb{A}}_\tau$.

To show this we need to make additional assumption on the step size $\tau$,
which we can later remove by meromorphic continuation.

Note that the definition of the parameter set $\mathcal{D}^{\bm k}_{M,N}$
does not depend on $\tau$. We can and will therefore restrict to
parameters $(\bm\ell,\eta)\in\mathcal{D}^{\bm k}_{M,N}$ 
and $\tau\in\mathbb{C}$ with $\Re(\tau)<0$
satisfying the additional conditions
\begin{equation} \label{tauinequality}
-\Re(\tau) \leq \min_{1 \le r \le N} \Re(\ell_r \eta) \text{ and }  - \Re(\tau) < \min_{1 \le r \le N} \frac{\Re \big( (2 \ell_r + 1 - n_{\bm k}(r))\eta\big)}{n_{\bm k}(r)+1}
\end{equation}
(which is possible since $(\bm\ell,\eta)\in\mathcal{D}^{\bm k}_{M,N}$).
We also take the base point $\bm\gamma$ in the restricted set 
$\Gamma^{\bm k}_{M,N;\tau}$ consisting of the 
base points $\bm\gamma\in\Gamma^{\bm k}_{M,N}$
satisfying
\begin{equation*}
\begin{split}
&-\Re(\ell_{k_i} \eta+\tau)<\Re(\gamma_{i})<\Re(\ell_{k_i} \eta+\tau),\\
&\Re(\gamma_{i(s+1;r)})+\Re(\eta-\tau)<\Re(\gamma_{i(s;r)})
\end{split}
\end{equation*}
for $1\leq i\leq M$, $1\leq r\leq N$ and $1\leq s<n_{\bm k}(r)$.
Note that $\Gamma^{\bm k}_{M,N;\tau}$ is nonempty 
as a consequence of $(\bm \ell,\eta) \in \mathcal{D}^{\bm k}_{M,N}$ and 
\eqref{tauinequality}.

We will use the proof \cite[Section 7]{RSV2} of the analogous statement for Jackson integral solutions of the boundary qKZ equations. 
These were defined as summations over $\bm x \in \bm x_0+\tau \Z^M$ for some
base point $\bm x_0$.
Hence each variable $x_j$ could be replaced by $x_j - \tau$ without affecting the overall value of the sum; such shifts provided a key step to the proof. 

To mimic the proof for Jackson integrals we 
will shift integration variables by $-\tau$ in the integrals
using Cauchy's theorem (recall that $\Re(\tau)<0$). We use the following 
standard observation. Let $f(\bm x)$ be a  
meromorphic function in $\bm x\in \C^M$ and view it, for fixed $x_i$ 
($i\not=j$), as meromorphic function in $x_j$.
Suppose it is $\pi\sqrt{-1}$-periodic in $x_j$ and 
holomorphic for $x_j$ in the vertical strip
\begin{equation} \label{strip} 
S_{\tau}(z) := \{ y\in\mathbb{C} \,\, | \,\, \Re(z)\leq\Re(x)\leq\Re(z-\tau) \}.
\end{equation}
Then
\[ \int_{z+[0,\pi] \sqrt{-1}} f(\bm x) \d x_j = \int_{z+[0,\pi] \sqrt{-1}} f(\bm x - \bm e_j \tau) \d x_j  \]
by a direct application of Cauchy's theorem.
So when in \cite{RSV2} a summation variable is shifted by $-\tau$ (this occurs in Lemma 7.6 in {\it ibid.}), we replace this by the procedure above based on the application of Cauchy's theorem.

The only other difference with the Jackson integral case arises because in the definition of $F_{\bm k}(\bm t)$, the contours depend on $\bm t$. In other
words, for $\bm t \in \widetilde{\mathbb{A}}_\tau$ we have
\[ F_{\bm k}(\bm t + \tau \bm e_r) =  \int_{C_{\bm k}^{\bm \gamma}(\bm t+\tau \bm e_r)} w_{\bm k}(\bm x;\bm t + \tau \bm e_r) \widetilde{\mathcal{B}}(\bm x;\bm t + \tau \bm e_r) \Omega \d^M \bm x, \]
with the shift in $\tau$ possibly appearing in one of the integration contours as well as the integrand.
If $r \ne k_j$ for all $j \in \{1,\ldots,M\}$, then the shift in $\tau$ does not affect the contour, hence
\begin{equation}\label{Psiwithtaushift}
F_{\bm k}(\bm t + \tau \bm e_r) =  \int_{C_{\bm k}^{\bm \gamma}(\bm t)} w_{\bm k}(\bm x;\bm t + \tau \bm e_r) \widetilde{\mathcal{B}}(\bm x;\bm t + \tau \bm e_r) 
\Omega \d^M \bm x. 
\end{equation}
We first claim that \eqref{Psiwithtaushift} is also true 
if $r=k_j$ for some $j \in \{1,\ldots,M\}$.
We can apply the above procedure 
involving Cauchy's theorem 
(without the subsequent variable substitution) 
with $z = t_{k_j} + \gamma_{j} + \tau$.
A subtlety arises if $n_{\bm k}(r)>1$, 
in which case Cauchy's theorem needs to be 
successively applied for all integration variables $x_j$ with $k_j=r$ 
starting with the variable $x_j$ with $j=i_{\bm k}(1;r)$ and 
working our way up to $i_{\bm k}(n_{\bm k}(r);r)$.
Hence, at each step we must assume that 
$x_i \in t_{k_i}+\gamma_{i} + [0,\pi] \sqrt{-1}$ if $k_i=k_j$ and $i<j$, and $x_i \in  t_{k_i}+\gamma_{i}+\tau + [0,\pi] \sqrt{-1}$ if $k_i=k_j$ and $i>j$.
To show that Cauchy's theorem indeed can be applied, we use the following
lemma. 

\begin{lem} \label{lem:integrandanalyticinstrip}
Let $\bm k\in I_{M,N}$. Let
$(\bm \ell,\eta) \in \mathcal{D}^{\bm k}_{M,N}$
and $\tau\in\C$ satisfying $\Re(\tau)<0$ and \eqref{tauinequality}. 
Choose $\bm \gamma \in \Gamma^{\bm k}_{M,N;\tau}$.
Let $1 \le j \le M$. Fix $\bm t \in \widetilde{\mathbb{A}}_\tau$, $x_i \in t_{k_i}+\gamma_{i} + [0,\pi] \sqrt{-1}$ if $i < j$ 
and $x_i \in t_{k_i} + \gamma_{i} + \delta_{k_i,k_j} \tau + [0,\pi] \sqrt{-1}$ 
if $i > j$, then
\[\frac{w_{\bm k}(\bm x;\bm t+ \tau \bm e_{k_j})}{\Phi_{\bm k}(\bm t+\tau\bm
e_{k_j})}\widetilde{\mathcal{B}}(\bm{x};\bm t+ \tau \bm e_{k_j}) \Omega \]
as a function of $x_j$ is $\pi \sqrt{-1}$-periodic and has no poles in $S_{\tau}(t_{k_j}+\gamma_{j} + \tau)$.
\end{lem}

\begin{proof}
The periodicity condition follows immediately from the analogous properties of $w_{\bm k}(\bm x;\bm t)$ and $\widetilde{\mathcal{B}}(\bm{x};\bm t) \Omega$. 
Using Lemma \ref{lem:poles}, the desired result on the poles holds if the real parts of all elements of $P^+_{\bm k;j}(\bm t + \tau \bm e_{k_j})$ are strictly greater than $\Re(t_{k_j}+\gamma_{j})$ and the real parts of all elements of $P^-_{\bm k;j}(\bm t+ \tau \bm e_{k_j})$ are strictly less than $\Re(t_{k_j}+\gamma_{j}+\tau)$. 
This imposes the inequalities
\[
\begin{array}{r@{\hspace{3pt}}ll} 
\Re(\gamma_{j}) &< \Re(\ell_{k_j} \eta + \tau), \\
-\Re(\gamma_{j}) &< \Re(\ell_{k_j} \eta), \\
\Re(\gamma_{i}- \gamma_{j}) &> \Re(\eta) , & \text{for } i<j \text{ and } k_i=k_j, \\ 
\Re(\gamma_{j}- \gamma_{i}) &> \Re(\eta) , & \text{for } i>j \text{ and } k_i=k_j, \\
\Re(t_s - t_{k_j}) &> \Re(\gamma_{j} - \ell_s \eta) , & \text{for } 1 \leq s < k_j, \\
\Re(t_{k_j} - t_s) &> \Re(-\gamma_{j} - \ell_s \eta - \tau) , & \text{for } k_j < s \leq N, \\
\Re(t_s+t_{k_j}) &> \Re(-\gamma_{j} - \ell_s \eta - \tau) , & \text{for } 1 \leq s \leq N, \\
\Re(t_{k_i} - t_{k_j}) &> \Re(\gamma_{j} - \gamma_{i} + \eta)  , & \text{for } i<j \text{ and } k_i<k_j, \\
\Re(t_{k_j} - t_{k_i}) &> \Re(\gamma_{i} -  \gamma_{j} + \eta-\tau)  , & \text{for } i>j \text{ and } k_i>k_j, \\
\Re(t_{k_i}+t_{k_j}) &> \Re(-\gamma_{i}- \gamma_{j} + \eta-\tau) , & \text{for } i \ne j \text{ and } k_i \ne k_j, \\
\Re(2t_{k_j}) &> \Re(-\gamma_{i}- \gamma_{j} + \eta-\tau) , & \text{for } i < j \text{ and } k_i = k_j, \\
\Re(2t_{k_j}) &> \Re(-\gamma_{i}- \gamma_{j} + \eta-2\tau) , & \text{for } i > j \text{ and } k_i = k_j, \\
\Re(t_{k_j}) &> \Re(-\widetilde \xi_+ - \tau  - \gamma_{j} ), \\
\Re(t_{k_j}) &> \Re(-\widetilde \xi_- - \tau - \gamma_{j}).
\end{array}
\]
The inequalities independent of $\bm t$ immediately follow from 
$\bm \gamma \in \Gamma_{M,N;\tau}^{\bm k}$  (the inequality $\Re(\gamma_{j}) < \Re(\ell_{k_j} \eta + \tau)$ is the only instance in this proof where we actually require $\bm \gamma \in \Gamma_{M,N;\tau}^{\bm k}$ as opposed to $\bm \gamma \in \Gamma_{M,N}^{\bm k}$).
For the inequalities involving $\bm t$ we also need the condition $|\Re(\gamma_{j})|<\Re(\ell_{k_j}\eta)$ and in addition conditions implied by $\bm t \in \widetilde{\mathbb A}_\tau$. 
Namely, the first three inequalities involving $\bm t$ follow from the conditions $\Re(t_r-t_{r+1})>2\Lambda_+ -\Re(\tau)$ for $1 \leq r < N$ and $\Re(t_N)>0$. 
The inequalities involving $t_{k_i} \pm t_{k_j}$ are a consequence of the conditions $\Re(t_r-t_{r+1}) > 2 \Lambda_+ + \Re(\eta-\tau)$, $\Re(t_N) > \Lambda_+ + \Re(\tfrac{\eta}{2}-\tau)$, $\Re(t_N)>0$ (and $\Re(\tau)<0$), and the final two inequalities are a consequence of $\Re(t_N) > \Lambda_+ + \max\bigl( \Re(-\widetilde \xi_+), \Re(-\widetilde \xi_- ) \bigr) - \Re(\tau)$. 
\end{proof}

Applying Lemma \ref{lem:integrandanalyticinstrip} for $j$ running through $\{ i \in \{1 ,\ldots,M\} \, | \, k_i = r\}$ from low to high values, we obtain
\eqref{Psiwithtaushift} 
Hence the two sides of the boundary qKZ equation \eqref{bqKZ} for $\Psi(\bm t) = F_{\bm k}(\bm t)$ have the same contour $C_{\bm k}^{\bm \gamma}(\bm t)$ and differ in the integrands in the same way as the summands did at the start of the proof in \cite[Section 7]{RSV2}.
Thus we follow that proof taking care to correctly use Cauchy's theorem when dealing with replacements $x_j \to x_j - \tau$.
We recall that the strategy is to shift the integration variable by $\tau$ in appropriate terms in \eqref{Psiwithtaushift}; these terms are essentially due to the expansion of $\widetilde{\mathcal{B}}(\bm x;\bm t)$ by means of \eqref{bBdecomposition}. 
Next, we use the conditions \eqref{wfunceqns} on the weight function $w_{\bm k}$ to establish that the terms on the left- and right-hand sides of \eqref{Psiwithtaushift} can be matched.

We need the following functional equations
\begin{align*}
F(x;\bm t + \tau \bm e_r) &= 
\e^{4 \ell_r \eta} \frac{\sinh(t_r \pm  x - \ell_r \eta + \tau)}{\sinh(t_r \pm  x + \ell_r \eta + \tau)} F(x;\bm t), \displaybreak[2] \\
F(x-\tau;\bm t) &= \biggl( \prod_{s=1}^N  \frac{\sinh(t_s - x -\ell_s \eta + \tau)}{\sinh(t_s - x + \ell_s \eta + \tau)} \frac{\sinh(t_s+ x+\ell_s \eta)}{\sinh(t_s+x-\ell_s \eta)} \biggr) F(x;\bm t), \displaybreak[2] \\
g(x-\tau) &= q^{-2} \e^{-2(\widetilde \xi_+ + \widetilde \xi_-)}   \frac{\sinh(x+\widetilde \xi_+)}{\sinh(x-\widetilde \xi_+ -\tau)} \frac{\sinh(x+\widetilde \xi_- )}{\sinh(x- \widetilde \xi_- - \tau)} g(x), \displaybreak[2] \\
h(x- \tau) &= \e^{2 \eta} \frac{\sinh(x- \tau)}{\sinh(x)} \frac{\sinh(x- \eta)}{\sinh(x+\eta-\tau)} h(x), \displaybreak[2]  \\
u_{\bm k;i}(x;\bm t+\tau \bm e_r) &= u_{\bm k;i}(x;\bm t) \times \begin{cases} 1 & \text{if } r<k_i, \\ \e^{2(\xi_++\xi_--\eta)+4(\sum_{s>k_i} \ell_s - M + i)\eta} & \text{if } r=k_i, \\ \e^{-4 \ell_r \eta} & \text{if } r > k_i, \end{cases} \displaybreak[2]  \\
u_{\bm k;i}(x-\tau;\bm t) &= q^2 \e^{2(\widetilde \xi_+ + \widetilde \xi_- )-4(M-i)\eta} u_{\bm k;i}(x;\bm t)
\end{align*}
where 
$k,r \in \{1,\ldots, N\}$ and $i \in \{1,\ldots,M\}$.
From these and \eqref{Phifunceqn} one derives
\begin{equation} \label{wfunceqns} \begin{aligned}
\frac{w_{\bm k}(\bm{x};\bm{t}+\tau \bm e_r)}{w_{\bm k}(\bm{x};\bm{t})} &= \biggl( \prod_{i=1}^M \frac{\sinh(t_r \pm x_i - \ell_r \eta + \tau)}{\sinh(t_r \pm x_i + \ell_r \eta + \tau)} \biggr) ,\\
\frac{w_{\bm k}(\bm{x}-\tau \bm e_j;\bm{t})}{w_{\bm k}(\bm{x};\bm{t})} &= \biggl( \prod_{s=1}^N \frac{\sinh(t_s + x_j + \ell_s \eta)}{\sinh(t_s + x_j - \ell_s \eta)} \frac{\sinh(t_s - x_j - \ell_s \eta + \tau)}{\sinh(t_s - x_j + \ell_s \eta + \tau)}  \biggr) \\
& \qquad \times \frac{\sinh(x_j + \widetilde \xi_+ )}{\sinh(x_j - \widetilde \xi_+ - \tau)} \frac{\sinh(x_j + \widetilde \xi_- )}{\sinh(x_j - \widetilde \xi_- -\tau )}  \\
& \qquad \times \biggl( \prod_{i \ne j} \frac{\sinh(x_j \pm x_i - \tau)}{\sinh(x_j \pm x_i)} \frac{\sinh(x_j \pm x_i - \eta)}{\sinh(x_j \pm x_i + \eta - \tau)} \biggr).
\end{aligned} \end{equation}
We emphasize that these are the same equations as those satisfied by the weight function $w(\bm x;\bm t)$ appearing in the Jackson integral in \cite[Thm. 6.2]{RSV2}; in particular they do not depend on $\bm k$.

We denote $e_r \bm t = (t_1,\ldots,t_{r-1},-t_r,t_{r+1},\ldots,t_N)$ for $\bm t = (t_1,\ldots,t_N) \in \C^N$ and $r \in \{1,\ldots,N\}$ and $J^{\c} = \{1,\ldots,M\} \setminus J$ for $J \subseteq \{1,\ldots,M\}$.
Define, for $\bm \epsilon \in \{ \pm \}^M$, $J \subseteq \{1,\ldots,M\}$, 
$r \in \{1,\ldots,N\}$ and $\bm x \in C_{\bm k}^{\bm \gamma}(\bm t)$:
\begin{align*}
m_r^{\bm \epsilon,J}(\bm x;\bm t) &:=
(-1)^{\# J^\c} \Biggl( \prod_{i =1}^M \epsilon_i \sinh(\widetilde \xi_- - \epsilon_i x_i) \prod_{s \ne r} \frac{\sinh(t_s - \epsilon_i x_i + \ell_s \eta)}{\sinh(t_s - \epsilon_i x_i -\ell_s \eta)} \Biggr) \\
& \qquad \times \biggl( \prod_{i' \in J^\c} \frac{1}{\sinh(t_r+\epsilon_{i'}  x_{i'} + \ell_r \eta)} \biggr) \biggl( \prod_{i \in J \atop i' \in J^\c} \frac{\sinh(\pm x_i + \epsilon_{i'} x_{i'} + \eta)}{\sinh(\pm x_i + \epsilon_{i'} x_{i'})} \biggr) \\
& \qquad \times \biggl( \prod_{i<i'\atop i,i'\in J \text{ or } i,i'\in J^\c} \frac{\sinh(\epsilon_i x_i + \epsilon_{i'} x_{i'} + \eta)}{\sinh(\epsilon_i x_i + \epsilon_{i'} x_{i'} )} \biggr),
\end{align*}
which defines a $\pi \sqrt{-1}$-periodic function of each $x_j$ with $j \in J^\c$;
this is the same function as in \cite[Eqn. (7.8)]{RSV2}. 

We require some further notation.
For $d \in \{0,\ldots,M\}$ consider a subset $J \subseteq \{1,\ldots,M\}$ 
of cardinality $M-d$. 
Write $J =  \{ i_1,\ldots,i_{M-d} \}$ with 
$1 \leq i_1<i_2<\ldots<i_{M-d} \leq M$.
Then for given $\bm k = (k_1,\ldots,k_M) \in I_{M,N}$, $\bm x = (x_1,\ldots,x_M) \in C_{\bm k}^{\bm \gamma}(\bm t)$ and $\bm \epsilon = (\epsilon_1,\ldots,\epsilon_M) \in \{ \pm \}^M$, we denote $\bm k_J := (k_{i_1},\ldots,k_{i_{M-d}})$, $\bm x_J := (x_{i_1},\ldots,x_{i_{M-d}})$ and $\bm \epsilon_J := (\epsilon_{i_1},\ldots,\epsilon_{i_{M-d}})$. 
Note that $\bm x_J \in C_{\bm k_J}(\bm t)$.

For the remainder of this subsection, we fix $r \in \{1,\ldots,N\}$, $d \in \{0,\ldots,M\}$, $J \subset \{1,\ldots,M\}$ such that $\# J = M-d$ and the subtuples $\bm k_J$, $\bm x_J$ and $\bm \epsilon_J$ as above.
For $\bm \epsilon_{J^\c} \in \{\pm \}^d$, set
\begin{align*}
\Lambda_{r,\bm \epsilon_{J^\c}}(\bm t) &:= \int_{C_{\bm k_{J^\c}}(\bm t)} w_{\bm k}(\bm x;\bm t) m_r^{\bm \epsilon,J}(\bm x;e_r \bm t) \d^d \bm x_{\bm J^\c}, \\
\Upsilon_{r,\bm \epsilon_{J^\c}}(\bm t) &:= \int_{C_{\bm k_{J^\c}}(\bm t)} w_{\bm k}(\bm x;\bm t) m_r^{\bm \epsilon,J}(\bm x;\bm t + \tau \bm e_r) \d^d \bm x_{\bm J^\c}.
\end{align*}
They are the analogons of their namesakes introduced in \cite[Sec. 7.3]{RSV2} with the summation over $\bm x_{\bm J^\c} \in \tau \Z^d$ replaced by an integral over $\bm x_{\bm J^\c} \in C_{\bm k_{J^\c}}(\bm t)$.

A careful inspection of the proof in \cite[Section 7]{RSV2} shows that the only statement which does not immediately generalize from the Jackson integral case to the integral case is in fact in the proof of \cite[Lemma 7.6]{RSV2} where in the Jackson integral a certain variable $x_j$ in the summand is replaced by $x_j - \tau$. 
In order to adapt that proof to one suitable for our integral solutions, 
we apply the argument based on Cauchy's theorem again.
Hence, the following lemma replaces the start of the
 proof of \cite[Lemma 7.6]{RSV2}.

\begin{lem} \label{lem:Zanalyticinstrip}
Let $\bm k\in I_{M,N}$. Let
$(\bm \ell,\eta) \in \mathcal{D}^{\bm k}_{M,N}$
and $\tau\in\C$ satisfying $\Re(\tau)<0$ and \eqref{tauinequality}. 
Choose $\bm \gamma \in \Gamma^{\bm k}_{M,N;\tau}$.
Let $j \in J^\c$ such that $\epsilon_j  = +$.
For fixed $\bm t \in \widetilde{\mathbb{A}}_\tau$ and 
$x_i \in t_{k_i}+\gamma_{i} + [0,\pi]\sqrt{-1}$ ($i \ne j$), the functions 
\begin{align*}
x_j & \mapsto w_{\bm k}(\bm x;\bm t) m_r^{\bm \epsilon,J}(\bm x;e_r \bm t), \\
x_j & \mapsto w_{\bm k}(\bm x;\bm t) m_r^{\bm \epsilon,J}(\bm x;\bm t + \tau \bm e_r)
\end{align*}
are $\pi \sqrt{-1}$-periodic and have no poles in $S_\tau(t_{k_j}+\gamma_{j})$.
\end{lem}

\begin{proof}
The periodicity statement follows from the fact that both $x_j \mapsto w_{\bm k}(\bm x;\bm t)$ and $x_j \mapsto  m_r^{\bm \epsilon,J}(\bm x; \bm t)$ are $\pi \sqrt{-1}$-periodic for any $\bm t$.

Using that $\epsilon_j = +$ we see that the poles of $m_r^{\bm \epsilon,J}(\bm x;e_r \bm t)$ as a function of $x_j$ are contained in the set
\[ \Bigl( \{ t_s - \ell_s \eta \, | \, 1 \leq s \leq N \}  \cup \{ \pm x_i \, | \, i \ne j \} \Bigr)+\pi \sqrt{-1} \Z. \]
Similarly, the poles of $m_r^{\bm \epsilon,J}(\bm x; \bm t + \tau \bm e_r)$ as a function of $x_j$ are contained in the set
\[ \Bigl( \{ t_s - \ell_s \eta  \, | \, 1 \leq s \leq N, \, s \ne r \} \cup \{ -t_r - \tau - \ell_r \eta \} \cup\{ \pm x_i \, | \, i \ne j \} \Bigr)+\pi \sqrt{-1} \Z. \]
Considering \eqref{wformula}, we see that all $x_j$-dependent poles are cancelled by zeros of $w_{\bm k}(\bm x;\bm t)$, as well as the poles of the form $t_s-\ell_s \eta$ for $s<k_j$.
It follows that all poles of $w_{\bm k}(\bm x;\bm t) m_r^{\bm \epsilon,J}(\bm x;e_r \bm t)$ and $w_{\bm k}(\bm x;\bm t) m_r^{\bm \epsilon,J}(\bm x;\bm t + \tau \bm e_r)$ are contained in
\[ \Bigl( (Q^+_{\bm k;j} - \tau \Z_{\geq 0}) \cup (Q^-_{\bm k;j} + \tau \Z_{\geq 0})  \Bigr)+\pi \sqrt{-1} \Z, \]
where
\begin{align*}
Q^+_{\bm k;j}(\bm t) &=  \{  t_s + \ell_s \eta  \}_{s \leq k_j } \cup \{ x_i - \eta  \}_{i<j} \\
Q^-_{\bm k;j}(\bm t) &= \{  t_s - \ell_s \eta \}_{s \geq k_j} \cup \{ -t_s -\ell_s \eta  -\tau \}_{s} \cup \{ -\widetilde \xi_+, -\widetilde \xi_- \} \cup \\
& \qquad  \cup \{ x_i + \eta \}_{i>j} \cup \{ - x_i + \eta \}_{i \ne j }.
\end{align*}
The conclusion of the lemma is justified if the real parts of all elements of $Q^+_{\bm k;j}$ are strictly greater than $\Re(t_{k_j} + \gamma_{j} - \tau)$ and the real parts of all elements of $Q^-_{\bm k;j}$ are strictly less than $\Re(t_{k_j} + \gamma_{j})$.
This yields the inequalities
\[
\begin{array}{r@{\hspace{3pt}}ll}  
\Re(\gamma_{j}) &> -\Re(\ell_{k_j}\eta) , \\
\Re(\gamma_{j}) &< \Re(\ell_{k_j} \eta + \tau),  \\
\Re(\gamma_{i}- \gamma_{j}) &> \Re(\eta-\tau), & \text{for } i<j \text{ and } k_i=k_j, \\ 
\Re(\gamma_{j}- \gamma_{i}) &> \Re(\eta-\tau), & \text{for } i>j \text{ and } k_i=k_j, \\
\Re(t_s - t_{k_j}) &> \Re(\gamma_{j} - \ell_s \eta - \tau), & \text{for } 1 \leq s < k_j, \\
\Re(t_{k_j} - t_s) &> \Re(-\gamma_{j} - \ell_s \eta), & \text{for } k_j < s \leq N, \\
\Re(t_s+t_{k_j}) &> \Re(-\gamma_{j} - \ell_s \eta -\tau), & \text{for } 1 \leq s \leq N, \\
\Re(t_{k_i} - t_{k_j}) &> \Re(\gamma_{j} -\gamma_{i}+ \eta - \tau) , & \text{for } i<j \text{ and } k_i<k_j, \\
\Re(t_{k_j} - t_{k_i}) &> \Re( \gamma_{i} -  \gamma_{j} + \eta) , & \text{for } i>j \text{ and } k_i>k_j, \\
\Re(t_{k_i}+t_{k_j}) &> \Re( -\gamma_{i}- \gamma_{j} + \eta), & \text{for } i \ne j \text{ and } k_i \ne k_j, \\
\Re(2t_{k_j}) &> \Re( -\gamma_{i}- \gamma_{j} + \eta), & \text{for } i \ne j \text{ and } k_i = k_j, \\
\Re(t_{k_j}) &> \Re(-\widetilde \xi_+  - \gamma_{j} ),  \\
 \Re(t_{k_j}) &> \Re( -\widetilde \xi_-  - \gamma_{j}).
\end{array}
\]
We get similar conditions as in the proof of Lemma \ref{lem:integrandanalyticinstrip}. 
However, the inequality for the differences on $\gamma_{j}$ is stronger as $\Re(\tau)<0$.
All these conditions are again a consequence of $\bm \gamma \in \Gamma_{M,N;\tau}^{\bm k}$, $\Re(\tau)<0$ and $\bm t \in \widetilde{\mathbb{A}}_\tau$. 
\end{proof}

In Fig. \ref{fig:complexplane} we illustrate a typical arrangement of poles with respect to the vertical strip $S_\tau(t_{k_i}+\gamma_i)$ as pertains to Lemma \ref{lem:Zanalyticinstrip}.\\

\begin{figure}[h]
\begin{tikzpicture}
\filldraw[fill=green!15,draw=white] (1.6+0.1,3.14) rectangle (1.6+0.1+1.3,-3.14);
\draw (1.6+0.1+1.3/2,2.8) node{$\scriptstyle S_\tau(t_1+\gamma_1)$};
\node (t1) at (1.6+0.1,-2.2) [circle,fill,inner sep=2pt,label={[label distance=-2pt]left:$\scriptstyle t_1+\gamma_1$}] {};
\node (t1+pi) at (1.6+0.1,-2.2+3.14) [circle,fill,inner sep=2pt,label={[label distance=-3pt]above:$\scriptstyle t_1+\gamma_1+\pi \sqrt{-1}$}] {};
\node (t1+pi+tau) at (1.6+0.1+1.3,-2.2-0.4+3.14) [circle,fill,inner sep=2pt,label={[label distance=-3pt]right:$\scriptstyle t_1 +\gamma_1- \tau + \pi \sqrt{-1}$}] {};
\node (t1+tau) at (1.6+0.1+1.3,-2.2-0.4) [circle,fill,inner sep=2pt,label={[label distance=-2pt]right:$\scriptstyle t_1 +\gamma_1- \tau$}] {};
\foreach \n in {0,1,2}
{
  \node at (1.6-1.5-1.3*\n,-2.2-0.4+0.4*\n)  {$\times$};
  \node at (1.6-1.5-1.3*\n,-2.2-0.4+3.14+0.4*\n) {$\times$};
}
\draw (1.6-1.5-.60,-2.2-0.3) node{$\scriptstyle t_1-\ell_1 \eta$};
\draw (1.6-1.5-1.10,-2.2-0.3+3.14) node{$\scriptstyle t_1-\ell_1 \eta+\pi \sqrt{-1}$};

\foreach \n in {0,1,2}
{
  \node at (1.6+1.5+1.3*\n,-2.2+0.4-0.4*\n) {$\times$};
  \node at (1.6+1.5+1.3*\n,-2.2+0.4+3.14-0.4*\n) {$\times$};
}
\draw (1.6+1.5+.50,-2.2+0.3+.25) node{$\scriptstyle t_1+\ell_1 \eta$};
\draw (1.6+1.5+1.00,-2.2+0.3+3.14+.25) node{$\scriptstyle t_1+\ell_1 \eta+\pi \sqrt{-1}$};

\foreach \n in {0}
{
  \node at (-1.6-1.5-1.3*\n,2.2-0.4+0.4*\n) {$\times$};
  \node at (-1.6-1.5-1.3*\n,2.2-0.4-3.14+0.4*\n) {$\times$};
}
\draw (-1.6-1.5,2.2-0.3+.20) node{$\scriptstyle -t_1-\ell_1 \eta$};
\draw (-1.6-1.5+0.50,2.2-0.3-3.14+.20) node{$\scriptstyle -t_1-\ell_1 \eta-\pi \sqrt{-1}$};

\draw[very thick,->] (t1) to (t1+pi);
\draw[->] (t1+pi) to (t1+pi+tau);
\draw[->] (t1+pi+tau) to (t1+tau);
\draw[->] (t1+tau) to (t1);
\draw[dotted] (0,-3.14) -- (0,3.14);
\draw[dotted] (-3.6,0) -- (6,0);
\end{tikzpicture}
\caption{Integration contour $t_1+\gamma_1 + [0,\pi] \sqrt{-1}$ and $t_1$-dependent poles of $x_1 \mapsto w_{\bm k}(\bm x ; \bm t) \widetilde{\mathcal{B}}(\bm x;\bm t) \Omega$ for $k_1 = 1$ and $\gamma_1$ satisfying $|\Re(\gamma_1)|< \Re(\ell_{k_1} \eta + \tau)$ as in the proof of Lemma \ref{lem:Zanalyticinstrip}. 
For each pole sequence, the pole closest to the vertical strip $S_\tau(t_1+\gamma_1)$ \eqref{strip}, which is the shaded area in the figure, is indicated.
The poles consist of (unilateral) sequences entirely to the left or entirely to the right of $S_\tau(t_1+\gamma_1)$.
In the vertical strip we have marked the closed contour with respect to which Cauchy's theorem is used in this subsection.}
\label{fig:complexplane}
\end{figure}
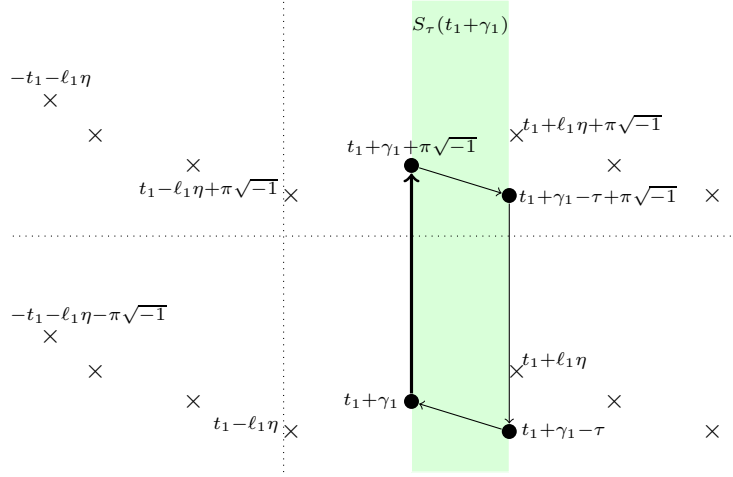

Let $\bm k\in I_{M,N}$ and $(\bm\ell,\eta)\in\mathcal{D}_{M,N}^{\bm k}$.
Following \cite{RSV2} it 
follows that for $\tau$ satisfying $\Re(\tau)<0$ and \eqref{tauinequality} 
that 
\begin{equation} \label{bqKZ:F}
F_{\bm k}(\bm t+ \tau \bm e_r) = \mathcal{A}_r(\bm t) F_{\bm k}(\bm t), \qquad  
1 \le r \le N
\end{equation}
as meromorphic functions in $\bm t\in \widetilde{\mathbb{A}}_\tau$. 

Next we show that the $V^{\bm\ell}(M)$-valued meromorphic 
function $F_{\bm k}=\Phi_{\bm k} \Theta_{\bm k}$ on 
$\widetilde{\mathbb{A}}_{\tau}$ uniquely extends to a meromorphic 
$V^{\bm\ell}(M)$-valued meromorphic function 
$\Psi_{\bm k}$ on $\C^N$ satisfying \eqref{bqKZ}.

Let $\{\mathcal{A}_{\alpha}(\bm t)\}_{{\alpha}\in\mathbb{Z}^N}$ be the unique 
family of linear operators on $V^{\bm\ell}$ 
depending meromorphically on $\bm t\in\mathbb{C}^N$ 
satisfying 
\begin{align*}
\mathcal{A}_{\alpha}(\bm t+{\beta}\tau)\mathcal{A}_{\beta}(\bm t)&=\mathcal{A}_{{\alpha}+{\beta}}(\bm t), && {\alpha},{\beta}\in\mathbb{Z}^N,\\
\mathcal{A}_{ 0}(\bm t)&=\textup{Id}_{V^{\bm\ell}},\\
\mathcal{A}_{\mathbf{e}_r}(\bm t)&=\mathcal{A}_r(\bm t), &&  r=1,\ldots,N,
\end{align*}
with $\mathcal{A}_r(\bm t)$ the transport operators 
\eqref{bqKZtransportmatrix} of the boundary qKZ equations (the compatibility
of the transport operators guarantees the existence of the cocycle
$\{\mathcal{A}_{\alpha}(\bm t)\}_{\alpha\in\mathbb{Z}^N}$). 

Consider the cone 
\[
P_+:=\{{ \lambda}\in\mathbb{Z}^N\,\, | \,\, \lambda_1\geq\lambda_2\cdots
\geq\lambda_N\geq 0\}.
\]
Note that $\bm t-{\lambda}\tau\in
\widetilde{\mathbb{A}}_\tau$ for $\bm t\in\widetilde{\mathbb{A}}_\tau$ and ${\lambda}\in P_+$.
It follows from \eqref{bqKZ:F}
that for all ${ \lambda}\in P_+$,
\begin{equation}\label{bqKZintegral}
F_{\bm k}(\bm t-{ \lambda}\tau)=\mathcal{A}_{-{ \lambda}}(\bm t)F_{\bm k}(\bm t)
\end{equation}
as meromorphic $V^{\bm\ell}(M)$-valued functions in 
$\bm t\in\widetilde{\mathbb{A}}_\tau$.
 
If $\Psi_{\bm k}$ is a $V^{\bm \ell}(M)$-valued meromorphic function 
on $\mathbb{C}^N$ satisfying the boundary qKZ equations and coinciding with 
$F_{\bm k}$ on $\widetilde{\mathbb{A}}_\tau$,
then 
\begin{equation}\label{extension}
\Psi_{\bm k}(\bm t)=\mathcal{A}_{\alpha}(\bm t-{\alpha}\tau)F_{\bm k}(\bm t-{\alpha}\tau)
\end{equation}
for $\bm t\in\mathbb{C}^N$ and 
${\alpha}\in\mathbb{Z}^N$ such that $\bm t-{\alpha}\tau\in
\widetilde{\mathbb{A}}_\tau$. Now \eqref{extension} can be used to prove 
the existence of the 
meromorphic extension of $F_{\bm k}$ to $\mathbb{C}^N$. For this we need to
show that the right hand side of \eqref{extension} does not depend on
the choice of ${\alpha}$ such that 
$\bm t-{\alpha}\tau\in\widetilde{\mathbb{A}}_\tau$. 

Suppose ${\beta}\in\mathbb{Z}^N$
also satisfies $\bm t-{\beta}\tau\in\widetilde{\mathbb{A}}_\tau$. 
Let ${ \lambda},{ \mu}\in P_+$ such that 
\[
{ \lambda}+\alpha={ \mu}+{\beta}\in P_+.
\]
Then repeated application of \eqref{bqKZintegral} gives
\begin{equation*}
\begin{split}
\mathcal{A}_{\beta}(\bm t-{\beta}\tau)F_{\bm k}(\bm t-{\beta}\tau)&=
\mathcal{A}_{\beta}(\bm t-{\beta}\tau)\mathcal{A}_{-{ \mu}}(\bm t-{\beta}\tau)^{-1}
F_{\bm k}(\bm t-({ \mu}+{ {\beta}})\tau)\\
&=\mathcal{A}_{\beta}(\bm t-{\beta}\tau)\mathcal{A}_{-{ \mu}}(\bm t-\beta\tau)^{-1}
\mathcal{A}_{-{ \lambda}}(\bm t-{\alpha}\tau)F_{\bm k}(\bm t-{\alpha}\tau)\\
&=\mathcal{A}_{\alpha}(\bm t-{\alpha}\tau)F_{\bm k}(\bm t-{\alpha}\tau),
\end{split}
\end{equation*}
where the last equality follows from the cocycle property. By a similar 
computation one shows that the resulting 
$V^{\bm\ell}(M)$-valued meromorphic function $\Psi_{\bm k}$
on $\mathbb{C}^N$ satisfies the boundary qKZ equations \eqref{bqKZ}.

The identity $\Psi_{\bm k}=F_{\bm k}$ on $\widetilde{\mathbb{A}}_\tau$
extends to $\widetilde{\mathbb{A}}$ by meromorphic continuation. Finally, 
the extra assumption \eqref{tauinequality} on the step size can be removed
by meromorphic continuation. This completes the proof of part {\bf b} of 
Theorem \ref{thm:main}.

\subsection{Asymptotics} 
\label{sec:asymptoticscompleteness}

In this subsection we prove part {\bf c} of Theorem \ref{thm:main}. 
Recall the notation $\bm t_{\bm k} := (t_{k_1},\ldots,t_{k_M}) \in \C^M$ for $\bm t \in \C^N$, $\bm k \in I_{M,N}$.
We will investigate the asymptotics of $\Psi_{\bm k}(\bm t)$ as 
$\bm t \stackrel{\mathbb{A}}{\to} \infty$
using the formula \eqref{integraltransformed}. 
It is convenient to study first the asymptotics of $\e^{\sum_i t_{k_i}} \frac{ w_{\bm k}(\bm t_{\bm k} + \bm y; \bm t)}{\Phi_{\bm k}(\bm t)}$ and of 
\[\e^{-\sum_i t_{k_i}} 
\widetilde{\mathcal{B}}(\bm t_{\bm k} + \bm y;\bm t) \Omega=
\e^{-\sum_it_{k_i}}\sum_{\bm m\in I_{M,N}}\beta_{\bm m}(\bm t_{\bm k}+\bm y; \bm t)
\Omega_{\bm m}.
\]

\begin{lem} \label{lem:wBasymptotics}
Let $\bm k \in  I_{M,N}$.
\begin{enumerate}
\item[{\bf a.}]
There exists unique continuous functions $\Delta_{\alpha}^{\bm k}$
on $C_{\bm k}^{\bm \gamma}(\bm 0)$ (${\alpha}\in Q_+$) such that 
\[
\e^{\sum_{i=1}^M t_{k_i}} \frac{w_{\bm k}(\bm t_{\bm k} + \bm y; \bm t)}
{\Phi_{\bm k}(\bm t)}=\sum_{{\alpha}\in Q_+}
\Delta_{\alpha}^{\bm k}(\bm y) \e^{-2({\alpha},\bm t)}
\]
for $(\bm y, \bm t)\in C_{\bm k}^{\bm \gamma}(\bm 0)\times\widetilde{\mathbb{A}}$, 
with the series converging normally on compact sets. 
Furthermore,
\begin{equation*}
\begin{split}
\Delta_0^{\bm k}(\bm y)&=  
\prod_{i=1}^M  \frac{\theta\bigl(\e^{2(y_i+\psi_{\bm k;i})}\bigr)}
{(\e^{2(y_i-\ell_{k_i} \eta)},q^2 \e^{-2(y_i + \ell_{k_i} \eta)};q^2)_\infty}\\
&\qquad\times 
\prod_{1\leq i<j\leq M \atop k_i=k_j} \bigl(1-\e^{2(y_j-y_i)}\bigr) 
\frac{(q^2 \e^{2(y_j-y_i-\eta)};q^2)_\infty}{(\e^{2(y_j-y_i+\eta)};q^2)_\infty}.
\end{split}
\end{equation*}
\item[{\bf b.}] Let $\bm m\in I_{M,N}$. There exists unique continuous 
functions $f_{\alpha}^{\bm k, \bm m}$ on $C_{\bm k}^{\bm \gamma}(\bm 0)$ (${\alpha}\in Q_+$) 
such that 
\[\e^{-\sum_{i=1}^M t_{k_i}}\beta_{\bm m}(\bm t_{\bm k}+\bm y;\bm t) 
=\sum_{{\alpha}\in Q_+}f_{\alpha}^{\bm k, \bm m}(\bm y) \e^{-({\alpha},\bm t)}
\]
for $(\bm y, \bm t)\in C_{\bm k}^{\bm \gamma}(\bm 0)\times\widetilde{\mathbb{A}}$, 
with the series converging
normally on compact sets. Furthermore,
\[
f_0^{\bm k, \bm m}(\bm y)=\delta_{\bm k,\bm m}
\prod_{i=1}^M  \frac{ - \e^{\xi_- + 2\sum_{s<k_i} \ell_s \eta}}{1-\e^{-2(y_i + \ell_{k_i} \eta)}}. \]
\end{enumerate}
\end{lem}

\begin{proof}
{\bf a.} This follows from the explicit expression
\begin{align*}
& \e^{\sum_{i=1}^M t_{k_i}} \frac{w_{\bm k}(\bm t_{\bm k} + \bm y; \bm t)}{\Phi_{\bm k}(\bm t)} =\\
&= \prod_{i=1}^M\Biggl\{\biggl(\prod_{s<k_i}
\frac{ ( \e^{2(y_i+t_{k_i} - t_s + \ell_s \eta)};q^2)_\infty}
{ ( \e^{2(y_i+t_{k_i} - t_s -\ell_s \eta)};q^2)_\infty}\biggr) 
\frac{ \theta\bigl(\e^{2(y_i+\psi_{\bm k;i})}\bigr) }
{(\e^{2(y_i-\ell_{k_i}\eta)},q^2 \e^{-2(y_i+\ell_{k_i}\eta)};q^2)_\infty}  \\
& \qquad\qquad \times \biggl( \prod_{s>k_i} 
\frac{ (q^2 \e^{2(t_s-t_{k_i}-y_i+\ell_s \eta)};q^2)_\infty}
{ (q^2 \e^{2(t_s-t_{k_i}-y_i-\ell_s \eta)};q^2)_\infty} \biggr) 
\biggl( \prod_{s=1}^N \frac{( \e^{-2(t_s+t_{k_i}+y_i-\ell_s \eta)};q^2)_\infty}
{( \e^{-2(t_s+t_{k_i}+y_i+\ell_s \eta)};q^2)_\infty} \biggr) \\
& \qquad\qquad \times  \frac{(q^2 \e^{2(\widetilde \xi_+ -t_{k_i}-y_i)},
q^2 \e^{2(\widetilde \xi_- -t_{k_i}-y_i)};q^2)_\infty}
{(\e^{-2(\widetilde \xi_++t_{k_i}+y_i)},
\e^{-2(\widetilde \xi_-+t_{k_i}+y_i)};q^2)_\infty} \Biggr\} \\
& \qquad \times 
\prod_{1\leq i<j\leq M}\Biggl\{  (1-\e^{-2( (y_i+t_{k_i})\pm (y_j+t_{k_j}))}) 
\frac{(q^2 \e^{-2((y_i+t_{k_i}) \pm (y_j+t_{k_j}) + \eta)};q^2)_\infty}
{(\e^{-2((y_i+t_{k_i}) \pm  (y_j+t_{k_j}) - \eta)};q^2)_\infty} \Biggr\}.
\end{align*}
{\bf b.} We use the explicit expression from Proposition
\ref{thm:Bethevectordecomposition}. It gives 
\begin{equation}\label{step1}
\begin{split}
&\e^{-\sum_i t_{k_i}}\beta_{\bm m}(\bm t_{\bm k}+\bm y;\bm t)=\\
&\e^{\sum_i \bigl(\tfrac{n_{\bm m}(m_i)}{2} - \ell_{m_i}\bigr) \eta}
\sum_{\bm p\in S_M\bm m}\,\,\sum_{\bm \epsilon \in \{\pm\}^M}\Biggl\{
\prod_{i=1}^M\Biggl\{\frac{  \epsilon_i  \e^{-t_{k_i}} 
\sinh( \epsilon_i t_{k_i} + \epsilon_i y_i  - \widetilde \xi_-)}
{\sinh( t_{p_i}+ \epsilon_i t_{k_i} +\epsilon_i y_i  - \ell_{p_i} \eta)}\\
& \hspace{3mm} \times  \biggl( \prod_{s>p_i} 
\frac{\sinh( t_s + \epsilon_i t_{k_i} +\epsilon_i y_i + \ell_s \eta)}
{\sinh(t_s + \epsilon_i t_{k_i} +\epsilon_i y_i - \ell_s \eta)} \biggr) 
\biggl( \prod_{s=1}^N \frac{\sinh(t_s - \epsilon_i y_i - \epsilon_i t_{k_i} + 
\ell_s \eta)}{\sinh(t_s - \epsilon_i y_i - \epsilon_i t_{k_i} - \ell_s \eta)}
 \biggr)  \Biggr\} \hspace{-8pt} \\
 & \qquad \times  \biggl( \prod_{1\leq i<j\leq M} \hspace{-2pt} 
\frac{\sinh(\epsilon_i t_{k_i}+  \epsilon_j t_{k_j} +\epsilon_i y_i + 
 \epsilon_j y_j + \eta)}{\sinh(\epsilon_i t_{k_i}+  \epsilon_j t_{k_j} +
\epsilon_i y_i +  \epsilon_j y_j )} \biggr) \\
 & \qquad \times  \biggl( \hspace{-2pt} \prod_{1\leq i<j\leq M\atop p_i<p_j} 
\hspace{-4pt} 
\frac{\sinh(\epsilon_i t_{k_i}- \epsilon_j t_{k_j} +\epsilon_i y_i  - 
\epsilon_j y_j - \eta)}{\sinh( \epsilon_i t_{k_i}- 
\epsilon_j t_{k_j} +\epsilon_i y_i  - \epsilon_j y_j )} \biggr)\biggr\}.
\end{split}
\end{equation}
For fixed $\bm p\in S_M\bm m$ and 
$\bm\epsilon\in\{\pm\}^M$, a direct computation 
shows that
\[
\prod_{i=1}^M\frac{  \epsilon_i  \e^{-t_{k_i}} 
\sinh( \epsilon_i t_{k_i} + \epsilon_i y_i  - \widetilde \xi_-)}
{\sinh( t_{p_i}+ \epsilon_i t_{k_i} +\epsilon_i y_i  - \ell_{p_i} \eta)}=
\sum_{{\alpha}\in Q_+}d_{\alpha}^{\bm p,\bm\epsilon}(\bm y) \e^{-({\alpha},\bm t)}
\]
with leading coefficient 
\[
d_0^{\bm p,\bm\epsilon}(\bm y)=\delta_{\bm p,\bm k}\delta_{\bm\epsilon,(-)^M}
\prod_{i=1}^M \frac{-\e^{\widetilde \xi_- -\ell_{k_i}\eta}}
{1-\e^{-2(y_i+\ell_{k_i}\eta)}} \]
and $(-)^M:=(-,\ldots,-)$ the $M$-tuple of minus signs. Here the convergence
is as indicated in the lemma.
For all the other factors in the right hand side of
 \eqref{step1} it is easy to compute the series expansion when 
$\bm t \stackrel{\mathbb{A}}{\to} \infty$. It leads to the result that
\[\e^{-\sum_{i=1}^M t_{k_i}}\beta_{\bm m}(\bm t_{\bm k}+\bm y;\bm t) 
=\sum_{{\alpha}\in Q_+}f_{\alpha}^{\bm k, \bm m}(\bm y) \e^{-({\alpha},\bm t)}
\]
with
\[
f_0^{\bm k, \bm m}(\bm y)=\delta_{\bm k, \bm m} \Biggl( \prod_{i=1}^M  \frac{-\e^{ \widetilde \xi_- + \bigl(\tfrac{n_{\bm k}(k_i)}{2} + 2 \sum_{s<k_i} \ell_s \bigr) \eta}}{1-\e^{-2(y_i+\ell_{k_i} \eta)}} \Biggr)  \e^{( \#\{ (i,j) \, | \, k_i < k_j \} - \#\{ (i,j) \, | \, i<j \}  ) \eta}
\]
and the convergence as indicated in the lemma.
Finally note that the expression of $f_0^{\bm k, \bm m}(\bm y)$ coincides
with the expression as given in the lemma because
\begin{equation*}
\begin{split}
\#\{ (i,j) \, | \, k_i<k_j \} - \#\{ (i,j) \, | \, i < j \} &= 
-\#\{ (i,j) \, | \, i < j, \, k_i=k_j \} \\
&  = \frac{M- \#\{ (i,j) \, |  \, k_i=k_j \} }{2}= \sum_{i=1}^M \frac{1-n_{\bm k}(k_i)}{2},
\end{split}
\end{equation*}
where we have used that $k_i \leq k_j$ if $i<j$.
\end{proof}

From Lemma \ref{lem:wBasymptotics} 
we immediately deduce that there exists unique continuous $V^{\bm\ell}(M)$-valued
functions $R_{\alpha}^{\bm k}$ on $C_{\bm k}^{\bm \gamma}(\bm 0)$  (${\alpha}\in Q_+$)
such that
\begin{equation}\label{almostthere}
\frac{w_{\bm k}(\bm t_{\bm k}+\bm y;\bm t)}{\Phi_{\bm k}(\bm t)}
\widetilde{\mathcal{B}}(\bm t_{\bm k}+\bm y;\bm t)\Omega=
\sum_{{\alpha}\in Q_+}R_{\alpha}^{\bm k}(\bm y) \e^{-({\alpha},\bm t)}
\end{equation}
for $(\bm y, \bm t)\in C_{\bm k}^{\bm \gamma}(\bm 0)\times\widetilde{\mathbb{A}}$
with the $V^{\ell}(M)$-valued series converging
normally on compact sets, and with leading coefficient given by
\begin{equation*}
\begin{split}
R_0^{\bm k}(\bm y)&=\sum_{\bm m\in I_{M,N}}\Delta_0^{\bm k}(\bm y)
f_0^{\bm k,\bm m}(\bm y)\Omega_{\bm m}\\
&=\Delta_0^{\bm k}(\bm y)\biggl( \prod_{i=1}^M  \frac{ - \e^{\xi_- + 2\sum_{s<k_i} \ell_s \eta}}{1-\e^{-2(y_i + \ell_{k_i} \eta)}}   \biggr)\Omega_{\bm k}\\
&=\Biggl( \prod_{i=1}^M \frac{-\e^{\xi_-+2\sum_{s<k_i} \ell_s \eta} \theta\bigl(\e^{2(y_i + \psi_{\bm k;i})}\bigr)}{(\e^{2(\pm y_i-\ell_{k_i} \eta)};q^2)_\infty} \Biggr)\\
&\qquad\qquad\times
\biggl(\prod_{1\leq i<j\leq M\atop k_i=k_j} \bigl(1-\e^{2(y_j-y_i)}\bigr) \frac{(q^2 \e^{2(y_j-y_i-\eta)};q^2)_\infty}{(\e^{2(y_j-y_i+\eta)};q^2)_\infty}\biggr)\Omega_{\bm k}.
\end{split}
\end{equation*}

It follows from \eqref{integraltransformed},
\eqref{almostthere} and Fubini's theorem 
that $\Theta_{\bm k}(\bm t)$ has a $V^{\bm\ell}(M)$-valued series expansion
\[
\Theta_{\bm k}(\bm t)=\sum_{{\alpha}\in Q_+}
L_{\alpha}^{\bm k} \e^{-({\alpha},\bm t)}
\]
normally converging for $\bm t$ in compact subsets of $\widetilde{\mathbb{A}}$ and 
with the coefficients given by
\[
L_{\alpha}^{\bm k}=\int_{C_{\bm k}^{\bm \gamma}(\bm 0)}R_{\alpha}^{\bm k}(\bm y)d^M\bm y,\qquad
{\alpha}\in Q_+.
\]
By the explicit expression of $R_0^{\bm k}(\bm y)$ we in particular have
\[
L_{0}^{\bm k}=\mu_{\bm k}\Omega_{\bm k}
\]
with $\mu_{\bm k}$ the explicit integral
\begin{equation} \label{muintegral}
\begin{aligned}
\mu_{\bm k}  &= \int_{C_{\bm k}^{\bm \gamma}(\bm 0)} \Biggl( \prod_{i=1}^M \frac{-\e^{\xi_-+
2\sum_{s<k_i} \ell_s \eta} \theta\bigl(\e^{2(y_i + \psi_{\bm k;i})}\bigr)}{(\e^{2(\pm y_i-\ell_{k_i} \eta)};q^2)_\infty} \Biggr) \\
& \hspace{15mm} \times \Biggl( \prod_{1\leq i<j\leq M \atop k_i=k_j} \bigl(1-\e^{2(y_j-y_i)}\bigr) \frac{(q^2 \e^{2(y_j-y_i-\eta)};q^2)_\infty}{(\e^{2(y_j-y_i+\eta)};q^2)_\infty} \Biggr)  \d^M \bm y.
\end{aligned}
\end{equation}
So to complete the proof of part {\bf c} of Theorem \ref{thm:main}, it suffices
to prove the following explicit evaluation formula for the integral
$\mu_{\bm k}$.

\begin{lem} \label{lem:mufactorization}
Let $\bm k \in  I_{M,N}$, $(\bm\ell,\eta)\in\mathcal{D}_{M,N}^{\bm k}$
and $\tau\in\mathbb{C}$ with $\Re(\tau)<0$.
Let $\bm \gamma \in \Gamma_{M,N}^{\bm k}$.
Then $\mu_{\bm k} = \nu_{\bm k}$, with $\nu_{\bm k}$ given by \eqref{nuproduct}.
\end{lem}

\begin{proof}
Fix $\bm k \in I_{M,N}$. Recall the notation $i_{\bm k}(m;r)$ from \eqref{ibmk}.
It follows that 
\[
\psi_{\bm k;i(m;r)} =  \omega_{\bm k;r} + \tfrac{\tau}{2} + (2m-n_{\bm k}(r)-1)\eta,
\]
with $\omega_{\bm k;r}$ given by \eqref{omega}.

Let $r \in \{1,\ldots,N\}$.
We introduce the meromorphic function $f_r^{\bm k}$ on $\C^{n_{\bm k}(r)}$ by
\begin{align*} f_r^{\bm k}(\bm z) &:= \Biggl( \prod_{m=1}^{n(r)} \frac{\theta(q\e^{2(z_m + \omega_{\bm k;r} + (2m-n(r)-1) \eta)})}{(\e^{2(\pm z_m-\ell_r \eta)};q^2)_\infty} \Biggr) \\
& \qquad \times  \Biggl( \prod_{1\leq m<m'\leq n_{\bm k}(r)}
\bigl(1-\e^{2(z_{m'}-z_m)}\bigr) \frac{(q^2 \e^{2(z_{m'}-z_m-\eta)};q^2)_\infty}
{(\e^{2(z_{m'}-z_m+\eta)};q^2)_\infty} \Biggr) 
\end{align*}
and the cycles
\[ \mathcal{C}_{(r)} := \bigl(\gamma_{i(1;r)}+\sqrt{-1}[0,\pi]\bigr)\times
\cdots\times\bigl(\gamma_{i(n_{\bm k}(r);r)}+\sqrt{-1}[0,\pi]\bigr).
\]
Define
\begin{align} 
\label{mur} \mu_r^{\bm k}&:= \int_{\mathcal{C}_{(r)}} f_r^{\bm k}(\bm z) 
\d^{n_{\bm k}(r)} \bm z \\
\label{nur} \nu_r^{\bm k} &:=  
\prod_{m=1}^{n_{\bm k}(r)}(-\pi \sqrt{-1}) \frac{(q^2\e^{-2m\eta}, q\e^{2((m-1-\ell_r)\eta \pm \omega_{\bm k;r})};q^2)_\infty}
{(q^2,q^2 \e^{-2\eta},\e^{2(m-1-2\ell_r)\eta};q^2)_\infty}.
\end{align}

Because 
\[ C_{\bm k}^{\bm\gamma}(\bm 0) = \mathcal{C}_{(1)} \times \cdots \times 
\mathcal{C}_{(N)} \]
and the product over $i<j$ in \eqref{muintegral} entails the restriction $k_i=k_j$, we note that the integral $\mu_{\bm k}$ factorizes into a product of 
$n_{\bm k}(r)$-fold integrals for $r \in \{1,\ldots,N\}$:
\begin{equation} \label{mufactorization}
\mu_{\bm k} = \biggl( \prod_{i=1}^M -\e^{\xi_-+2\sum_{s<k_i} \ell_s \eta} \biggr) 
\prod_{r=1}^N \mu_r^{\bm k}.
\end{equation}
In view of \eqref{nuproduct} it therefore suffices to show that 
$\mu_r^{\bm k}  = \nu_r^{\bm k}$. 

We prove $\mu_r^{\bm k}  = \nu_r^{\bm k}$
for the following parameter values. We fix $\bm k\in I_{M,N}$,
$r\in\{1,\ldots,N\}$ with $n_{\bm k}(r)\geq 1$ 
and $\tau\in\mathbb{C}$ with $\Re(\tau)<0$. Consider
the path-connected non-empty parameter domain
\[
\mathcal{D}=\Bigl\{ (\ell_r,\eta)\in\mathbb{C}^2\,\, | \,\, 
\Re(\ell_r\eta)> \max\Bigl(0,\tfrac{n_{\bm k}(r)-1}{2} \Re(\eta) \Bigr)\Bigr\}.
\]
For fixed $(\ell_r,\eta)\in\mathcal{D}$ let
\[
\bm\gamma_{(r)}=\bigl(\gamma_{i(1;r)},\ldots,\gamma_{i(n_{\bm k}(r);r)}\bigr)\in
\mathbb{C}^{n_{\bm k}(r)}
\] 
be such that 
\begin{equation*}
\begin{split}
&-\Re(\ell_r\eta)<\Re(\gamma_{i(s;r)})<\Re(\ell_r\eta),\\
&\Re(\gamma_{i(s^\prime+1;r)}) \leq \Re(\gamma_{i(s^\prime;r)}),\\
&\Re(\gamma_{i(s^\prime+1;r)})+\Re(\eta)<\Re(\gamma_{i(s^\prime;r)})
\end{split}
\end{equation*}
for $1\leq s\leq n_{\bm k}(r)$ and $1\leq s^\prime<n_{\bm k}(r)$.
Note that if $\Re(\eta)<0$ then $\bm\gamma_{(r)}=(0,\ldots,0)$ satisfies the conditions.   

The integral $\mu_r^{\bm k}$ (see \eqref{mur})
does not depend on the choice of $\bm\gamma_{(r)}$. This follows
from the fact that the separation of the poles of $f_r^{\bm k}(\bm z)$
by the cycle $\mathcal{C}_{(r)}$ does not depend on the choice of $\bm\gamma_{(r)}$,
and Cauchy's theorem. It now also follows directly that
$\mu_r^{\bm k}$ is holomorphic in $(\ell_r,\eta)\in\mathcal{D}$.
The same is true for $\nu_r^{\bm k}$ (see \eqref{nur}) by a direct inspection. 
Hence it suffices to prove that $\mu_r^{\bm k}=\nu_r^{\bm k}$ for the 
restricted parameter domain 
\[\{(\ell_r,\eta)\in\mathcal{D} \,\, | \,\,
\Re(\eta)<0\}=\{(\ell_r,\eta)\in\mathbb{C}^2\,\, | \,\,
\Re(\ell_r\eta)>0\,\,\&\,\, \Re(\eta)<0\}.
\]
But this is the special case of 
\cite[Appendix D - Proof of formula (5.13)]{TV2}
with the associated parameters specialized to 
\[ 
a=b=\e^{-2\ell_r \eta}, \qquad c=q \, \e^{-2\omega_{\bm k;r}}, 
\qquad p = q^2 , \qquad x=\e^{2\eta}. \hfill 
\qedhere
\]
\end{proof}

\subsection{Completeness}

Consider the boundary qKZ equations \eqref{bqKZ} as a compatible system of 
difference equations for $V^{\bm\ell}(M)$-valued meromorphic functions on
$\C^N$. The fact that the leading
coefficients $\nu_{\bm k}\Omega_{\bm k}$ of $\Theta_{\bm k}(\bm t)$
as $\bm t \stackrel{\mathbb{A}}{\to} \infty$ ($\bm k\in I_{M,N}$) 
form a linear basis of $V^{\bm\ell}(M)$ 
implies that the $\{\Psi_{\bm k}\}_{\bm k\in I_{M,N}}$
form a linear basis of 
the space of $V^{\bm\ell}(M)$-valued meromorphic solutions of the boundary qKZ
equations over the field
of $\tau\mathbb{Z}^N$-periodic meromorphic functions on $\C^N$
by the arguments of \cite[\S 5.6]{vMS}. This proves part {\bf d} of 
Theorem \ref{thm:main}.

\section{Integral solutions for finite-dimensional representations 
of quantum $\mathfrak{sl}_2$} \label{sec:findim}

Fix $\ell_r\in\frac{1}{2}\mathbb{Z}_{>0}$ for $1\leq r\leq N$ and write
\[
\pr^{\bm\ell}:=\pr^{\ell_1}\otimes\cdots\otimes\pr^{\ell_N}:
V^{\bm\ell}\rightarrow 
\overline{V}^{\bm\ell}:=\overline{V}^{\ell_1}\otimes\cdots\otimes
\overline{V}^{\ell_N}
\]
for the projection onto the finite-dimensional quotient.
Write 
\[ \widebar{\Omega}_{\bm k} := \pr^{\bm\ell}\Omega_{\bm k} = 
\overline{v}^{\ell_1}_{n_{\bm k}(1)} \otimes \cdots \otimes 
\overline{v}^{\ell_N}_{n_{\bm k}(N)}  \]
for $\bm k \in I_{M,N}$. Note that
\[
\overline{V}^{\bm\ell}=\bigoplus_{M=0}^{2\sum_{r=1}^N\ell_r}\overline{V}^{\bm\ell}(M)
\]
with
\[
\overline{V}^{\bm\ell}(M)=\bigoplus_{\bm k\in I_{M,N}^{\bm\ell}}
\mathbb{C}\overline{\Omega}_{\bm k}
\]
and with index set
\[
I_{M,N}^{\bm\ell}:=\{\bm k\in I_{M,N} \,\, | \,\,
n_{\bm k}(r)\leq 2\ell_r\quad \forall\, r=1,\ldots,N\}.
\]

Let $\overline{\mathcal{A}}_r(\bm t): \overline{V}^{\bm\ell}\rightarrow
\overline{V}^{\bm\ell}$ be the linear operators such that
\[
\overline{\mathcal{A}}_r(\bm t)\circ\pr^{\bm\ell}=
\pr^{\bm\ell}\circ\mathcal{A}_r(\bm t),
\]
with $\mathcal{A}_r(\bm t): V^{\bm\ell}\rightarrow V^{\bm\ell}$ the transport operators of the boundary qKZ
equations \eqref{bqKZ}. The boundary qKZ equations for 
$\overline{V}^{\bm\ell}$-valued meromorphic functions $\overline{\Psi}$
on $\C^N$ are
\begin{equation}\label{bqKZfin}
\overline{\Psi}(\bm t+\tau\bm e_r)=\overline{\mathcal{A}}_r(\bm t)
\overline{\Psi}(\bm t),\qquad r\in\{1,\ldots,N\}.
\end{equation}

Let $\eta\in\mathbb{C}$ with $\Re(\eta)>0$. Then for all 
$\bm k\in I_{M,N}^{\bm\ell}$ we have $(\bm\ell,\eta)\in\mathcal{D}_{M,N}^{\bm k}$.
Hence Theorem \ref{thm:main} immediately gives the following result.
\begin{thm}\label{thm:findim}
Let $\tau\in\mathbb{C}$ with $\Re(\tau)<0$.
Let $\ell_r\in\frac{1}{2}\mathbb{Z}_{>0}$ ($1\leq r\leq N$) and 
$\eta\in\C$ with $\Re(\eta)>0$. Let $0\leq M\leq 2\sum_{r=1}^{N}\ell_r$.
For $\bm k\in I_{M,N}^{\bm\ell}$ let $\Psi_{\bm k}$ be the $V^{\bm\ell}(M)$-valued 
meromorphic solution of the boundary qKZ equations \eqref{bqKZ} with respect
to the parameters $(\bm\ell,\eta)$, as defined in
Theorem \ref{thm:main}. 

Then $\{\pr^{\bm\ell}\Psi_{\bm k}\}_{\bm k\in I_{M,N}^{\bm\ell}}$ is a basis
of the $\overline{V}^{\bm\ell}(M)$-valued meromorphic solutions of the
boundary qKZ equations \eqref{bqKZfin}
over the field of $\tau\mathbb{Z}^N$-invariant 
meromorphic functions.
\end{thm}

\begin{rema}
If $\ell_1 = \ldots = \ell_N = \tfrac{1}{2}$ then 
$\widebar{V}^{\bm \ell} \cong (\C^2)^{\otimes N}$ 
and $\widebar{V}^{\bm \ell}(M)$ is spanned by vectors 
$\widebar{\Omega}_{\bm k}$ with $\bm k \in I_{M,N}$ such that $n_{\bm k}(r) \leq 1$ for all $r \in \{1,\ldots,N\}$. 
For all such $\bm k$ we may take $\bm \gamma = (0,0,\ldots,0) \in 
\Gamma_{M,N}^{\bm k}$, since $k_i=k_j$ for $i \ne j$ does not occur. 
\end{rema}

\appendix

\section{Asymptotics of the boundary qKZ equations} \label{app:bqKZasymptotics}

\begin{lem} \label{lem:RKinfinity}
Let $\ell_1,\ell_2 \in \C$.
The limits
\begin{equation} \label{RKinfty}
R_\infty^{\ell_1 \ell_2}:= \lim_{\Re(x) \to \infty} R^{\ell_1 \ell_2}(x), \qquad K_\infty^{\ell_1}(\xi) := \lim_{\Re(x) \to \infty} K^{\ell_1}(x;\xi).
\end{equation}
exist and for all $d_1,d_2 \in \Z_{\geq 0}$ we have
\begin{equation}
\label{RKinftyformula}
\begin{aligned}
R_\infty^{\ell_1 \ell_2} (v^{\ell_1}_{d_1} \otimes v^{\ell_2}_{d_2}) &= \e^{2(d_1 d_2 - \ell_1 d_2 - d_1 \ell_2) \eta} v^{\ell_1}_{d_1} \otimes v^{\ell_2}_{d_2} \\
K_\infty^{\ell_1}(\xi) (v^{\ell_1}_{d_1}) &=  (-1)^{d_1} \e^{-d_1 \left( 2\xi+(2\ell_1-d_1)\eta \right)} v^{\ell_1}_{d_1}.
\end{aligned}
\end{equation}
\end{lem}

\begin{proof}
Let $\mathcal{R}$ be the truncated universal R-matrix for $\widehat{\mathcal{U}}_\eta$, see Section \ref{sec:uniR}. 
After twisting it by $z^\frac{d}{2}\otimes f_1$ we will have an element $\mathcal{R}(z) \in \widehat{\mathcal{U}}_\eta^{\otimes 2}[[z]]$ of the form $\mathcal{R}(z)=\exp(\frac{\eta}{2} h_1 \otimes h_1)(1+O(z))$.
The evaluation of $\mathcal{R}(z)$ in the tensor product of two irreducible representations $V^{\ell_1}\otimes V^{\ell_2}$
gives the R-matrix $\e^{2\ell_1\ell_2\eta} R^{\ell_1 \ell_2}(x)$ where $z=\e^x$; the extra factor is due to the normalized action on the tensor product of highest weight vectors.
The limit $\Re(x)\to -\infty$ corresponds to $z\to 0$.
Thus, for the R-matrices in question we have:
\[
R^{\ell_1 \ell_2}(x)\to \e^{2 \ell_1 \ell_2 \eta} (\pi^{\ell_1 } \otimes \pi^{\ell_2}) \exp \left(-\tfrac{\eta}{2} h_1\otimes h_1 \right)
\]
as $\Re(x)\to -\infty$, and owing to unitarity \eqref{unitarity} we obtain
\[
R^{\ell_1 \ell_2}(x)\to \e^{-2 \ell_1 \ell_2 \eta}   (\pi^{\ell_1 } \otimes \pi^{\ell_2}) \exp \left(\tfrac{\eta}{2} h_1\otimes h_1\right)
\]
as $\Re(x)\to \infty$.
When applied to $v^{\ell_1}_{d_1} \otimes v^{\ell_2}_{d_2}$, this gives the desired formula.

The asymptotic formula for $K^\ell(x;\xi)$ follows immediately from \eqref{Kdiagonal}.
\end{proof}

As a consequence of Lemma \ref{lem:RKinfinity} we have
\begin{lem} \label{lem:Ainfinity}
Let $\bm \ell \in \C^N$ and $r \in \{1,\ldots,N\}$.
Then the limit $\mathcal{A}_{\infty,r}$ as defined in \eqref{Ainfty} exists and for all $\bm k \in I_{M,N}$ and all $r \in \{1,\ldots,N\}$ we have
\begin{equation} \label{Ainftyformula}  \mathcal{A}_{\infty,r} (\Omega_{\bm k}) = \varphi_{\bm k;r}  \Omega_{\bm k},
\end{equation}
with $ \varphi_{\bm k;r}$ given by \eqref{varphidefn}.
\end{lem}

\begin{proof} 
According to Lemma \ref{lem:RKinfinity}, the desired limit exists and is equal to
\begin{align*} 
\mathcal{A}_{\infty,r} &= (R^{\ell_r \, \ell_{r+1}}_{\infty})_{r \, r+1} \cdots (R^{\ell_r \, \ell_N}_\infty)_{r \, N} \Bigl(K^{\ell_r}_\infty(\xi_+)\Bigr)_r (R^{\ell_N \, \ell_r}_\infty)_{N \, r} \cdots (R^{\ell_{r+1} \, \ell_r}_\infty)_{r+1 \, r}\\
& \qquad \times (R^{\ell_{r-1} \, \ell_r}_\infty)_{r-1 \, r} \cdots (R^{\ell_1 \, \ell_r}_\infty)_{1 \, r} \Bigl( K^{\ell_r}_\infty(\xi_-)\Bigr)_r  \bigl( R^{\ell_1 \, \ell_r}_\infty \bigr)^{-1}_{1 \, r} \cdots \bigl(R^{\ell_{r-1} \, \ell_r}_\infty \bigr)^{-1}_{r-1 \, r} \\
&= \bigl(K^{\ell_r}_\infty(\xi_+) K^{\ell_r}_\infty(\xi_-)\bigr)_r  \prod_{s > r} \bigl( R_\infty^{\ell_r \ell_s} \bigr)^2_{r\, s},
\end{align*}
where we have used \eqref{Psymm} and the diagonality of both $R_\infty^{\ell_r \, \ell_s}$ and $K_\infty^{\ell}(\xi)$.
Applying \eqref{RKinftyformula}, for $\bm d \in \Z_{\geq 0}^N$ we have
\begin{align*}
 \mathcal{A}_{\infty,r} (v_{\bm d}) &=
\e^{-2d_r\bigl(\xi_++\xi_-+(2\ell_r-d_r)\eta\bigr)} \biggl( \prod_{s>r} \e^{4 \bigl(d_r d_s - \ell_r d_s - d_r \ell_s \bigr) \eta} \biggr)   v_{\bm d} \\
&= \e^{2d_r (\eta - \xi_+ - \xi_- ) + 4 d_r \Bigl(  \sum_{s>r}d_s+\tfrac{d_r-1}{2} - \sum_{s \geq r} \ell_s\Bigr)\eta } \Bigl( \prod_{s>r} \e^{-4 \ell_r d_s \eta}  \Bigr) v_{\bm d}.
\end{align*}
Using the bijection $\zeta_{M,N}: I_{M,N} \to P_N(M)$ we arrive at the following formula 
\begin{align*}
 \mathcal{A}_{\infty,r} (\Omega_{\bm k}) &=  \e^{2n_{\bm k}(r) (\eta - \xi_+ - \xi_- ) + 4 n_{\bm k}(r) \Bigl(\sum_{s>r}n_{\bm k}(s)+\tfrac{n_{\bm k}(r)-1}{2} - \sum_{s \geq r} \ell_s  \Bigr)\eta } \\
& \hspace{60mm} \times  \Bigl( \prod_{s>r} \e^{-4 \ell_r n_{\bm k}(s) \eta}  \Bigr) \Omega_{\bm k}
\end{align*}
for $\bm{k}\in I_{M,N}$.
We denote $\max_{\bm k}(r) =  \max\{ i \in \{1,\ldots,M\} \, | \, k_i=r \}$.
Recalling the definition of $n_{\bm k}(s)$ as given by \eqref{ndefn} we have $\sum_{s>r} n_{\bm k}(s) =  M - \textstyle \max_{\bm k}(r)$.
From the formula for the sum of a finite arithmetic progression we infer that
\[ n_{\bm k}(r)  \biggl( \sum_{s>r}n_{\bm k}(s)+\frac{n_{\bm k}(r)-1}{2} \biggr) 
=\sum_{i=1 \atop k_i=r}^M (M-i). \]
Formula \eqref{Ainftyformula} now readily follows.
\end{proof}

\section{The boundary Bethe vectors} \label{app:Bethevectors}

In this section we prove Proposition
\ref{thm:Bethevectordecomposition}

\subsection{Decomposition of ordinary Bethe vectors}

For generic values of $\bm \ell,\bm t \in \C^N$ and $x \in \C$ we write
\[ T^{\bm \ell}(x;\bm t) = L^{\ell_1}_{0 \, 1}(x-t_1) \cdots L^{\ell_N}_{0 \, N}(x-t_N) = \begin{pmatrix} \ast & B^{\bm \ell}(x;\bm t) \\ \ast & D^{\bm \ell}(x;\bm t) \end{pmatrix}  \]
with $B^{\bm \ell}(x;\bm t), D^{\bm \ell}(x;\bm t) \in \End(V^{\bm \ell})$; for the case $N=0$ the convention on empty products means $B^{\emptyset}(x;\emptyset)=0$.
For $\bm \ell, \bm t \in \C^N$ and $\bm x \in \C^M$, we define
\[ B^{\bm \ell}(\bm x;\bm t) = B^{\bm \ell}(x_1;\bm t) \cdots B^{\bm \ell}(x_M;\bm t) \in \End(V^{\bm \ell}) \]
with the convention that for $M=0$ we have $B^{\bm \ell}(\emptyset;\bm t) = \Id_{V^{\bm \ell}}$.
From \eqref{RTT} it follows that $[B^{\bm \ell}(x;\bm t) , B^{\bm \ell}(y;\bm t)]=0$; as a consequence we have $B^{\bm \ell}(w \bm x;\bm t) \Omega = B^{\bm \ell}(\bm x;\bm t) \Omega$ for any permutation $w \in S_M$.
The vectors $B^{\bm \ell}(\bm x;\bm t) \Omega$ are called \emph{spin-$\bm \ell$ ordinary (or type A) Bethe vectors}.
A simple induction argument with respect to $N$, ultimately a consequence 
of the ice rule for the $R$-operators,
establishes that
\begin{equation} \label{Bdecomposition}
B^{\bm \ell}(\bm x;\bm t) \Omega = \sum_{\bm d \in P_N(M)} b^{\bm \ell}_{\bm d}(\bm x;\bm t) v_{\bm d}
\end{equation}
for some $b^{\bm \ell}_{\bm d}(\bm x;\bm t) \in \C$, depending meromorphically on $\bm x \in \C^M$ and $\bm t \in \C^M$.\\

It is the aim of this section to provide a self-contained prescription for the closed formula for the coefficients $b^{\bm \ell}_{\bm d}(\bm x;\bm t)$.
They are the ``trigonometric weight funtions'' of \cite{TV2}; we refer to this work, and references therein, for other points of view on $b^{\bm \ell}_{\bm d}(\bm x;\bm t)$ and other derivations for their explicit formulae.\\

Recall that $J^\c := \{1,\ldots,M\} \setminus J$ for any $J \subset \{1,\ldots,M\}$ and introduce the notations
$\widehat{\bm z} := (z_1,\ldots,z_{N-1}) \in \C^{N-1}$ for $\bm z = (z_1,\ldots,z_N) \in \C^N$ and $\bm x_J := (x_{i_1},\ldots,x_{i_{M-d}})$ for $\bm x = (x_1,\ldots,x_M) \in \C^M$ and $J = \{ i_1,\ldots,i_{M-d} \}$ with $i_1<i_2<\ldots<i_{M-d}$.
From a statement analogous to \cite[Lemma 4.2]{RSV} it follows that
\begin{equation} \label{Brecursive} \begin{aligned}
\hspace{-4pt} B^{\bm \ell}(\bm x;\bm t) \Omega &=\sum_{J \subset \{1,\ldots,M\} } \biggl( \prod_{i \in J} \frac{\sinh(t_N - x_i - (\tfrac{1}{2} - \ell_N)\eta)}{\sinh(t_N - x_i -(\tfrac{1}{2} + \ell_N)\eta)}  \biggr) \\
& \quad \times \biggl( \prod_{i \in J \atop j \in J^\c} \frac{\sinh(x_i-x_j+\eta)}{\sinh(x_i-x_j)} \biggr) \biggl(B^{\widehat{\bm \ell}}({\bm x}_J;\widehat{\bm t})  \otimes \prod_{i \in J^\c} B^{\ell_N} (x_i;t_N) \biggr) \Omega. \hspace{-8pt}
\end{aligned} \end{equation}
Cf. \cite[Eq. (5.6)]{RSV2} we have, for $\ell, x, t \in \C$ and $d \in \Z_{\geq 0}$,
\[ B^\ell(x;t) v_d^\ell = \frac{-\e^{(-\ell+\tfrac{1}{2}+d)\eta} \sinh(\eta)}{\sinh(t-x-(\tfrac{1}{2}+\ell)\eta)} v_{d+1}^\ell. \]
Combining this with \eqref{Brecursive} we obtain
\begin{equation} \label{Brecursion0} \begin{aligned}
\hspace{-10pt} B^{\bm \ell}(\bm x;\bm t) \Omega &=\sum_{d=0}^M \e^{d(\frac{d}{2}-\ell_N)\eta}  \sum_{J \subset \{1,\ldots,M\} \atop \# J^\c =d}  \biggl( \prod_{i \in J^\c} \frac{-\sinh(\eta)}{\sinh(t_N -x_i- (\tfrac{1}{2} + \ell_N)\eta)} \biggr) \\
 & \hspace{24mm} \times \Biggl( \prod_{i \in J} \frac{\sinh(t_N - x_i - (\tfrac{1}{2} - \ell_N)\eta)}{\sinh(t_N - x_i - (\tfrac{1}{2} + \ell_N)\eta)} \Biggr) \\
 & \hspace{24mm} \times \Biggl( \prod_{i \in J \atop j \in J^\c} \frac{\sinh(x_i-x_j+\eta)}{\sinh(x_i-x_j)} \Biggr)  \Bigl(  B^{\widehat{\bm \ell}}(\bm x_J ; \widehat{\bm t}) \Omega^{\widehat{\bm \ell}} \Bigr) \otimes v^{\ell_N}_d. \hspace{-10pt}
\end{aligned} \end{equation}

A collection of functions
\[ \{  c^{\bm \ell}_{\bm d}: \C^M \otimes \C^N \to \C \text{ meromorphic} \, | \, \bm \ell \in \C^N, \, \bm d \in P_N(M) \}_{M,N \in \Z_{\geq 0}} \]
is said to satisfy \emph{quantum affine $\mathfrak{sl}_2$-recursion}
if the initial condition
\begin{equation} \label{Binitial} c^{\emptyset}_\emptyset(\emptyset;\emptyset) = 1 \end{equation}
and the recurrence relation
\begin{equation}  \label{Brecursion} \begin{aligned}
c^{\bm \ell}_{\bm d}(\bm x;\bm t)  &=\e^{d_N(\frac{d_N}{2}-\ell_N)\eta}  \sum_{J \subset \{1,\ldots,M\} \atop \# J^\c = d_N}  \biggl( \prod_{i \in J^\c} \frac{-\sinh(\eta)}{\sinh(t_N-x_i-(\tfrac{1}{2} + \ell_N)\eta)} \biggr)  \\
 & \hspace{40mm} \times \biggl( \prod_{i \in J} \frac{\sinh(t_N - x_i - (\tfrac{1}{2} - \ell_N)\eta)}{\sinh(t_N - x_i - (\tfrac{1}{2} + \ell_N)\eta)} \biggr) \\
 & \hspace{40mm} \times \biggl( \prod_{i \in J \atop j \in J^\c} \hspace{-2pt} \frac{\sinh(x_i-x_j+\eta)}{\sinh(x_i-x_j)} \biggr)  c^{\widehat{\bm \ell}}_{\widehat{\bm d}}(\bm x_J;\widehat{\bm t})\hspace{-13pt}
\end{aligned} \end{equation}
hold true.
Our plan is now as follows.
First we will show that the $  b_{\bm d} $ implicitly defined in \eqref{Bdecomposition} satisfy this recursion.
Then we will present another collection of meromorphic functions $a^{\bm \ell}_{\bm d} $  on $\C^M \times \C^N$, parametrized by $\bm \ell \in \C^N$, $\bm d \in \Z_{\geq 0}^N$, by means of a closed-form expression, which can be shown to satisfy the same recursion.
Since quantum affine $\mathfrak{sl}_2$-recursion has a unique solution, we obtain $b^{\bm \ell}_{\bm d} = a^{\bm \ell}_{\bm d}$ and an explicit expression for the Bethe vectors in terms of the basis $\{ v^{\bm \ell}_{\bm d} \}_{\bm d \in \Z^N_{\geq 0}}$ follows.

\begin{lem}  \label{lem:recursion1}
$ \{  b^{\bm \ell}_{\bm d}\, | \, \bm \ell \in \C^N, \, \bm d \in P_N(M)  \}_{M,N \in \Z_{\geq 0}}$ satisfies quantum affine $\mathfrak{sl}_2$-recursion.
\end{lem}

\begin{proof}
In the case $N=0$, from $B^\emptyset(\bm x;\emptyset) = \delta_{0,M}$ we obtain \eqref{Binitial}.
Combining \eqref{Brecursion0} with \eqref{Bdecomposition} and using that the $v^{\bm \ell}_{\bm d}$ form a basis for $V^{\bm \ell}$ we obtain that the $b^{\bm \ell}_{\bm d}$ also satisfy the recurrence relation \eqref{Brecursion}.
\end{proof}

Let $\bm d \in P_N(M)$ and define
\[ I(\bm d) := \{ \bm m \in \{1,\ldots,N\}^M \, | \, \forall s\, n_{\bm m}(s) = d_s \} \]
and
\begin{align*}
a^{\bm \ell}_{\bm d}(\bm x;\bm t) &= \sum_{\bm m \in I(\bm d)} \biggl( \prod_i \frac{-\e^{\bigl(\tfrac{n_{\bm m}(m_i)}{2}  - \ell_{m_i}\bigr)\eta} \sinh(\eta)}{\sinh( t_{ m_i} - x_i - (\tfrac{1}{2} + \ell_{m_i})\eta)}  \prod_{s>m_i}  \frac{\sinh(t_s - x_i - (\tfrac{1}{2} - \ell_s)\eta)}{\sinh(t_s - x_i - (\tfrac{1}{2} + \ell_s)\eta)} \biggr) \\
& \hspace{70mm} \times  \prod_{i,j  \atop m_i<m_j}\frac{\sinh(x_i-x_j+\eta)}{\sinh(x_i-x_j)}.
\end{align*}

\begin{lem} \label{lem:recursion2}
$\{ a^{\bm \ell}_{\bm d} \, | \, \bm \ell \in \C^N, \, \bm d \in P_N(M) \}_{M,N \in \Z_{\geq 0}}$ satisfies quantum affine $\mathfrak{sl}_2$-recursion.
\end{lem}

\begin{proof}
It is immediately seen that the $ a^{\bm \ell}_{\bm d}$ satisfy \eqref{Binitial} owing to the convention that empty sums are zero and empty products are one.
To establish \eqref{Brecursion} for the $ a^{\bm \ell}_{\bm d}$ first note that if $n_{\bm m}(s) = d_s$ then
\[  \sum_{i=1}^M \Bigl(\frac{n_{\bm m}(m_i)}{2}  - \ell_{m_i}\Bigr)  = \sum_{s=1}^N \sum_{i=1 \atop m_i =s}^M  \Bigl(\frac{n_{\bm m}(s)}{2}  - \ell_s\Bigr) = \sum_{s=1}^N d_s \Bigl(\frac{d_s}{2}  - \ell_s\Bigr). \]
Note that
\begin{align*}
a^{\widehat{\bm \ell}}_{\widehat{\bm d}}(\bm x_{ J};\widehat{\bm t}) =&
\e^{\sum_{s=1}^{N-1} d_s\bigl(\tfrac{d_s}{2} - \ell_s\bigr) \eta}  \sum_{\bm m  \in I(\widehat{\bm d})} \Biggl( \prod_{i \in J} \frac{-\sinh(\eta)}{\sinh(t_{ m_i} - x_i - (\tfrac{1}{2} + \ell_{ m_i })\eta)}  \\
& \hspace{10mm} \times  \prod_{s =m_i+1}^{N-1} \frac{\sinh(t_s -x_i -  (\tfrac{1}{2} - \ell_s)\eta)}{\sinh(t_s - x_i - (\tfrac{1}{2} + \ell_s)\eta)} \Biggr) \prod_{i,j\in J \atop m_i<m_j} \frac{\sinh(x_i-x_j+\eta)}{\sinh(x_i-x_j)}.
\end{align*}
Hence, the right-hand side of \eqref{Brecursion} for $c^{\bm \ell}_{\bm d} = a^{\bm \ell}_{\bm d}$ is given by
\begin{align*}
& \e^{\sum_{s=1}^N d_s\bigl(\tfrac{d_s}{2} - \ell_s\bigr) \eta} \sum_{\bm m  \in I(\widehat{\bm d})}  \sum_{J \subset \{1,\ldots,M\} \atop \# J^\c = d_N}  \\
& \hspace{15mm} \biggl( \prod_{i \in J^\c} \frac{-\sinh(\eta)}{\sinh(t_N -x_i - (\tfrac{1}{2} + \ell_N)\eta)} \biggr) \biggl( \prod_{i \in J}   \frac{-\sinh(\eta)}{\sinh( t_{m_i} - x_i -  (\tfrac{1}{2} + \ell_{m_i} )\eta)} \biggr) \\
& \hspace{15mm} \times \Biggl( \prod_{i \in J} \frac{\sinh(t_N - x_i -  (\tfrac{1}{2} - \ell_N)\eta)}{\sinh(t_N - x_i - (\tfrac{1}{2} + \ell_N)\eta)}  \prod_{s =m_i+1 }^{N-1} \frac{\sinh(t_s - x_i - (\tfrac{1}{2} - \ell_s)\eta)}{\sinh(t_s - x_i - (\tfrac{1}{2} + \ell_s)\eta)}  \Biggr)    \\
& \hspace{15mm} \times \biggl( \prod_{i \in J \atop j \in J^\c} \frac{\sinh(x_i-x_j+\eta)}{\sinh(x_i-x_j)} \biggr)\biggl(  \prod_{i,j\in J \atop m_i<m_j} \frac{\sinh(x_i-x_j+\eta)}{\sinh(x_i-x_j)} \biggr).
\end{align*}
We specify a map
\[  I(\widehat{\bm d}) \times \{ J \subset \{1,\ldots,M\} \, | \, \#J = M-d_N \}  \to I(\bm d) \]
by inserting $N$'s in the $(M-d_N)$-tuples at the places given by the elements of $J^\c = \{1,\ldots,M\} \setminus J$, i.e.
\begin{align*} ( (m_1,\ldots,m_{M-d_N}),J) \mapsto & (m_1,\ldots,m_{j_1-1},N,m_{j_1},\ldots,m_{j_2-2},N,m_{j_2-1},\ldots, \\
& \qquad \qquad \ldots,m_{j_{d_N}-d_N},N,m_{j_{d_N}-d_N+1},\ldots,m_{M-d_N}), \end{align*}
where $J^\c  = \{ j_1,\ldots,j_{d_N}\}$ with $j_1<j_2<\ldots < j_{d_N}$; this is evidently injective, and since both sets have finite cardinality $\frac{(M-d_N)!}{\prod_{s =1}^{N-1} d_s!} \binom{M}{d_N} = \frac{M!}{\prod_{s=1}^N  d_s!}$ it follows that the map is bijective.
Hence the right-hand side of \eqref{Brecursion} equals
\begin{align*}
&  \e^{\sum_{s=1}^N d_s\bigl(\tfrac{d_s}{2} - \ell_s\bigr) \eta}  \sum_{\bm m \in I(\bm d)}   \\
& \qquad \biggl( \prod_{i=1 \atop m_i = N}^M \frac{-\sinh(\eta)}{\sinh(t_N - x_i - (\tfrac{1}{2} + \ell_N)\eta)} \biggr) \biggl( \prod_{i=1 \atop m_i \ne N}^M \frac{-\sinh(\eta)}{\sinh( t_{m_i} - x_i - (\tfrac{1}{2} + \ell_{m_i})\eta)} \biggr) \\
& \qquad \times \Biggl( \prod_{i=1 \atop m_i \ne N}^M \frac{\sinh(t_N - x_i - (\tfrac{1}{2} - \ell_N)\eta)}{\sinh(t_N - x_i - (\tfrac{1}{2} + \ell_N)\eta)} \prod_{s=m_i+1}^{N-1} \frac{\sinh(t_s - x_i - (\tfrac{1}{2} - \ell_s)\eta)}{\sinh(t_s - x_i - (\tfrac{1}{2} + \ell_s)\eta)}  \Biggr)    \\
& \qquad \times \biggl( \prod_{i,j=1 \atop m_i < m_j = N}^{M} \frac{\sinh(x_i-x_j+\eta)}{\sinh(x_i-x_j)} \biggr)\biggl(  \prod_{i,j =1 \atop m_i<m_j \ne N}^M \frac{\sinh(x_i-x_j+\eta)}{\sinh(x_i-x_j)} \biggr),
\end{align*}
where the remaining products recombine so as to yield $a^{\bm \ell}_{\bm d}(\bm x;\bm t)$ as required.
\end{proof}

Combining Lemmas \ref{lem:recursion1} and \ref{lem:recursion2} we deduce
\begin{equation}  \label{Bdecomposition2} \begin{aligned}
B^{\bm \ell}(\bm x;\bm t) \Omega  = \hspace{-3pt} \sum_{\bm k \in I_{M,N}} \sum_{\bm m \in S_M(\bm k)} & \biggl(  \prod_{i =1}^M  \frac{-\e^{\bigl(\tfrac{n_{\bm m}(m_i)}{2}  - \ell_{m_i}\bigr)\eta} \sinh(\eta)}{\sinh(t_{m_i} - x_i - (\tfrac{1}{2} + \ell_{m_i})\eta)} \\
& \qquad \times \prod_{s=m_i+1}^N \frac{\sinh(t_s - x_i - (\tfrac{1}{2} - \ell_s)\eta)}{\sinh(t_s - x_i - (\tfrac{1}{2} + \ell_s)\eta)} \biggr) \hspace{-14pt}  \\
 &   \times   \biggl( \prod_{i,j =1 \atop m_i<m_j}^M  \frac{\sinh(x_i-x_j+\eta)}{\sinh(x_i-x_j)} \biggr)   \Omega_{\bm m}. \hspace{-14pt}
\end{aligned} \end{equation}

\subsection{Decomposition of the boundary Bethe vectors} \label{app:bBethevectordecomposition}

Having completed the induction, we may now consider $\bm \ell$ fixed and drop it from the notation.
Because of \eqref{Lcrossing} we have
\[  T(-x;\bm t)^{-1}  = \biggl( \prod_{s=1}^N \frac{\sinh(t_s+x+(\tfrac{1}{2}-\ell_s)\eta)}{\sinh(t_s+x+(\tfrac{1}{2}+\ell_s)\eta)}\biggr) \siy_0 T(-x-\eta;\bm t)^{t_0} \siy_0. \]
Hence, from \eqref{bB0defn} it follows that
\begin{equation} \label{bBexpression0}
\begin{aligned}
\hspace{-7pt} \mathcal{B}(x;\bm t) &=  \biggl( \prod_{s=1}^N \frac{\sinh(t_s+x+(\tfrac{1}{2}-\ell_s)\eta)}{\sinh(t_s+x+(\tfrac{1}{2}+\ell_s)\eta)}\biggr) \\
& \quad \times \biggl( D(-x-\eta;\bm t) B(x;\bm t) - \frac{\sinh(\xi_--x)}{\sinh(\xi_-+x)} B(-x-\eta;\bm t) D(x;\bm t) \biggr). \hspace{-10pt}
\end{aligned}
\end{equation}
From \eqref{RTT} we obtain the commutation relation
\begin{align*}
& D(x;\bm t) B(y;\bm t) = \\
&= \frac{\sinh(x-y+\eta)}{\sinh(x-y)} B(y;\bm t) D(x;\bm t)  - \frac{\sinh(\eta)}{\sinh(x-y)} B(x;\bm t) D(y;\bm t).
\end{align*}
Using this and a trigonometric identity we infer from \eqref{bBexpression0} that
\begin{align*}
\mathcal{B}(x;\bm t)   &= \frac{\sinh(2x)}{\sinh(2x+\eta)}  \biggl( \prod_{s=1}^N \frac{\sinh(t_s+x+(\tfrac{1}{2}-\ell_s)\eta)}{\sinh(t_s+x+(\tfrac{1}{2}+\ell_s)\eta)}\biggr)  \\
& \quad \times \biggl( B(x;\bm t)D(-x-\eta;\bm t)  - \frac{\sinh(\xi_--x-\eta)}{\sinh(\xi_-+x)} B(-x-\eta;\bm t) D(x;\bm t) \biggr).
\end{align*}
Hence
\begin{equation} \label{bBexpression1} \widetilde{\mathcal{B}}( x;\bm t) = \sum_{\epsilon \in \{\pm \}} \epsilon \frac{\sinh(\widetilde \xi_- - \epsilon x)}{\sinh(\eta)} B(-\epsilon x-\tfrac{\eta}{2};\bm t) D(\epsilon x - \tfrac{\eta}{2};\bm t). \end{equation}
\eqref{bBexpression1} serves as the base case of an inductive argument analogous to the proof of \cite[Prop. 4.1]{RSV}.
It establishes that
\begin{equation} \label{bBexpression2}  \begin{aligned}
\widetilde{\mathcal{B}}(\bm x;\bm t) =& \sum_{\bm \epsilon \in \{\pm\}^M } \biggl( \prod_{i=1}^M \epsilon_i  \frac{\sinh(\widetilde \xi_- - \epsilon_i x_i)}{\sinh(\eta)} \biggr) \biggl( \prod_{i,j = 1 \atop i<j}^M \frac{\sinh(\epsilon_i x_i + \epsilon_j x_j + \eta)}{\sinh(\epsilon_i x_i + \epsilon_j x_j)} \biggr) \\
 & \hspace{20mm} \times \biggl( \prod_{i=1}^M B(-\epsilon_i x_i - \tfrac{\eta}{2};\bm t) \biggr) \biggl( \prod_{i=1}^M D(\epsilon_i x_i - \tfrac{\eta}{2};\bm t) \biggr). \end{aligned} \end{equation}

In \cite[Eq. (5.8)]{RSV2} the relation
\begin{equation} \label{DOmega} D(x;\bm t) \Omega = \left( \prod_{s=1}^N\frac{\sinh(t_s-x-(\tfrac{1}{2}-\ell_s)\eta)}{\sinh(t_s-x-(\tfrac{1}{2}+\ell_s)\eta)} \right) \Omega \end{equation}
is derived.
Combining it with \eqref{bBexpression2} we find the following expression for the spin-$\bm \ell$ boundary Bethe vectors in terms of the spin-$\bm \ell$ ordinary (type A) Bethe vectors:
\begin{equation}  \label{bBformula}  \begin{aligned}
\hspace{-12pt} \widetilde{\mathcal{B}}(\bm x;\bm t) \Omega =& \sum_{\bm \epsilon \in \{\pm\}^M } \Biggl( \prod_{i=1}^M \epsilon_i  \frac{\sinh(\widetilde \xi_- - \epsilon_i x_i )}{\sinh(\eta)}  \prod_{s=1}^N \frac{\sinh(t_s - \epsilon_i x_i + \ell_s \eta)}{\sinh(t_s - \epsilon_i x_i - \ell_s \eta)} \Biggr)  \\
 & \qquad  \times \biggl( \prod_{i,j=1 \atop i<j}^M \frac{\sinh(\epsilon_i x_i + \epsilon_j x_j + \eta)}{\sinh(\epsilon_i x_i + \epsilon_j x_j)} \biggr) \biggl( \prod_{i=1}^M B(-\epsilon_i x_i - \tfrac{\eta}{2};\bm t) \biggr) \Omega. \hspace{-15pt}
\end{aligned} \end{equation}
In this formula we can then substitute expression \eqref{Bdecomposition2} 
to arrive at an explicit formula for the coefficients 
$\beta_{\bm k}(\bm x;\bm t)$ in the decomposition 
\eqref{bBdecomposition}, yielding Thm. \ref{thm:Bethevectordecomposition}.


\end{document}